\documentclass[12pt]{amsart}
\usepackage{}

\usepackage{cancel}
\usepackage{amsmath}
\usepackage{amsfonts}
\usepackage{amssymb}
\usepackage[all]{xy}           %xypic macro for latex2.09

\usepackage{bbding}
\usepackage{txfonts}
\usepackage{amscd}

\usepackage[shortlabels]{enumitem}
\usepackage{ifpdf}
\ifpdf
  \usepackage[colorlinks,final,backref=page,hyperindex]{hyperref}
\else
  \usepackage[colorlinks,final,backref=page,hyperindex,hypertex]{hyperref}
\fi
\usepackage{tikz}
\usepackage[active]{srcltx}

\makeatletter

\newtheorem{thm}{Theorem}[section]
\newtheorem{lem}[thm]{Lemma}
\newtheorem{cor}[thm]{Corollary}
\newtheorem{pro}[thm]{Proposition}
\newtheorem{ex}[thm]{Example}
\newtheorem{rmk}[thm]{Remark}
\newtheorem{defi}[thm]{Definition}

\newcommand {\emptycomment}[1]{}

\newcommand{\dTD}{\dM^{\Theta_\huaD}}

\newcommand{\lon }{\,\rightarrow\,}
\newcommand{\be }{\begin{equation}}
\newcommand{\ee }{\end{equation}}

\newcommand{\g}{\mathfrak g}

\newcommand{\huaB}{\mathcal{B}}%{{\mathcal{E}}}%{\mathcal{B}}

\newcommand{\huaD}{\mathcal{D}}

\newcommand{\huaH}{\mathcal{H}}

\newcommand{\frkk}{\mathfrak k}

\newcommand{\frkC}{\mathfrak C}

\newcommand{\frkK}{\mathfrak K}
\newcommand{\hk}{\widehat{\mathfrak K}}\newcommand{\pk}{\mathfrak K'}

\newcommand{\dM}{\mathrm{d}}

\newcommand{\EXP}{\mathrm{Exp}}

\newcommand{\Id}{{\rm{Id}}}

\newcommand{\br}[1]{   [ \cdot,    \cdot  ]   }

\newcommand{\Hom}{\mathrm{Hom}}

\newcommand{\Ad}{\mathrm{Ad}}
\newcommand{\Aut}{\mathrm{Aut}}

\newcommand{\gl}{\mathfrak {gl}}

\begin{document}

\title{Cohomologies of difference Lie groups and van Est theorem}

\author{Jun Jiang}
\address{Department of Mathematics, Jilin University, Changchun 130012, Jilin, China}
\email{jiangjun20@mails.jlu.edu.cn}

\author{Yunnan Li}
\address{School of Mathematics and Information Science, Guangzhou University, Guangzhou 510006, China}
\email{ynli@gzhu.edu.cn}

\author{Yunhe Sheng}
\address{Department of Mathematics, Jilin University, Changchun 130012, Jilin, China}
\email{shengyh@jlu.edu.cn}

%\date{\today}

\begin{abstract}
A difference Lie group is a Lie group equipped with a difference operator, equivalently a crossed homomorphism with respect to the adjoint action. In this paper, first we introduce the notion of a representation of a difference Lie group, and establish the relation between representations of difference Lie groups and representations of difference Lie algebras via differentiation and integration. Then we introduce a cohomology theory for difference Lie groups and justify it via the van Est theorem. Finally, we classify abelian extensions of difference Lie groups using the second cohomology group as applications.
\end{abstract}


\renewcommand{\thefootnote}{}

\maketitle

\vspace{-1cm}
\tableofcontents

\allowdisplaybreaks

%\end{document}

\section{Introduction}

%\subsection{Difference groups}

Crossed homomorphisms on groups  first appeared in Whitehead's earlier work ~\cite{Whi}  and were later applied to study non-abelian Galois cohomology~\cite{Se}. Recently, crossed homomorphisms were used to study Hopf-Galois structures \cite{Tsang} and construct representations of mapping class groups of surfaces \cite{CS,Kasahara}.
In the definition of a crossed homomorphism $\huaD:G\to H$ on groups, there is an action of the  group $G$ on the group $H$. A crossed homomorphism on a   group $G$ with respect to the adjoint action is called a difference operator in this paper. Meanwhile, a   group together with a difference operator is called a difference   group. Difference operators on groups were studied in ~\cite{GLS}
as the inverse of Rota-Baxter operators on groups introduced there with the  motivation from factorization problems and integrable systems~\cite{RS1,STS}. In the category of Hopf algebras, similar structures are called bijective 1-cocycles, and   applied to construct solutions of the quantum Yang-Baxter equation~\cite{AGV}.

In the Lie algebra context, the concept of crossed homomorphisms  was introduced in \cite{Lue} in the study of non-abelian extensions of Lie algebras. Crossed homomorphisms on Lie algebras with respect to the adjoint representation are in fact difference operators (also called differential operators of weight 1), abstracted from original instance in numerical analysis to algebraic settings of associative and Lie algebras~\cite{GK,Lev,LGG}. Lie's third theorem holds for crossed homomorphisms on Lie groups and Lie algebras \cite{GLS,JS,MQ}. Crossed homomorphisms have various applications, e.g. in the study of post-Lie algebras and post-Lie Magnus expansion \cite{MQ} and representations of Cartan type Lie algebras \cite{PSTZ}.

In this paper, we study representations and cohomologies of difference Lie groups, and give applications in the study of abelian extensions of difference Lie groups.

A representation of a Lie group is a smooth homomorphism from this Lie group to the general linear Lie group of a vector space.
A basic tool to study representations of Lie groups is the usage of the corresponding ``infinitesimal'' representations of Lie algebras.
The representation theory of connected compact Lie groups parallels to
that of semisimple Lie algebras. %Unitary representations of Lie groups are widely studied.
We introduce the notion of representations of difference Lie groups and establish the relation with representations of difference Lie algebras via differentiation and integration. More precisely, one can obtain a representation of the difference Lie algebra by  differentiating  a representation of a difference Lie group, and conversely one can also obtain a representation of the difference Lie group by integrating a representation of a difference Lie algebra.

A classical approach to study a mathematical structure is to associate to it invariants. Among these, cohomology theories occupy a  central position as they enable for example to control deformations or extension problems. The cohomology theory of difference Lie algebras was given in \cite{JS}, and it was shown that infinitesimal deformations of a difference Lie algebra are classified by the second cohomology group. In this paper, we establish the cohomology theory for difference Lie groups with coefficients in arbitrary representations. To justify its correctness, we show that the van Est theorem holds for cohomologies of difference Lie groups and cohomologies of difference Lie algebras given in \cite{JS}. The classical van Est isomorphism \cite{Van}  gives the relation between the differentiable cohomology of Lie groups and the cohomology of Lie algebras. See \cite{AC,CD,Cr, Hou, Li,MS, PPT, RS}
for various van Est type theorems  and applications.

Finally we study abelian extensions of difference Lie groups as applications. We show that  abelian extensions of difference Lie groups are classified by the second cohomology group given above. As a byproduct, we classify difference operators on the semidirect product Lie group via certain quotient of the second cohomology group of the difference operator. See \cite{Ne} for more details of abelian extensions of infinite dimensional Lie groups.

The paper is organized as follows. In Section \ref{sec:rep}, we introduce the notion of   representations of  difference Lie groups, and establish its relation with representations of difference Lie algebras via differentiation and integration. In Section \ref{sec:coh}, we introduce a cohomology theory for difference Lie groups. To do that, first we give the cohomology of a difference operator, and then combine the cohomology of a difference operator and the cohomology of a Lie group to obtain the cohomology of a difference Lie group. The relation between these cohomology groups are given by a long exact sequence (Theorem \ref{cohomology-exact-DG}). In Section \ref{sec:van}, we show that the van Est theorem holds for cohomologies of difference Lie groups and cohomologies of difference Lie algebras. In Section \ref{sec:ext}, we classify abelian extensions of difference Lie groups in terms of the second cohomology group introduced in Section \ref{sec:coh}.

\vspace{2mm}
\noindent
{\bf Acknowledgements. } This research is supported by the National Natural Science Foundation of China (Grant Nos. 11922110, 12071094) and Guangdong Basic and Applied Basic Research Foundation (Grant No. 2022A1515010357).

\section{Representations of difference Lie groups and difference Lie algebras}\label{sec:rep}
In this section, we introduce the notion of a representation of a difference Lie group, and establish the relation between representations of difference Lie groups and representations of difference Lie algebras via differentiation and integration.
%{\color{blue}Not necessary to fix a ground field?}

\subsection{Representations of difference Lie groups}

We introduce the notion of a representation of a difference Lie group, which gives rise to the semidirect product difference Lie group.
\begin{defi}
Let $G$ be a Lie group.  A smooth map $\huaD: G\lon G$ is called  a {\bf difference operator} on $G$ if the following equality holds:
\begin{equation}\label{defiD}
\huaD(gh)=\huaD(g)g\huaD(h)g^{-1}, \quad \forall g, h\in G.
\end{equation}
 A {\bf difference Lie group}  $(G, \huaD)$ is a Lie group $G$ equipped with a difference operator $\huaD$.
\end{defi}

\begin{ex}{\rm
Let $G$ be an abelian Lie group. Then a Lie group homomorphism $\huaD: G\lon G$ is a difference operator.}
\end{ex}

\begin{ex}{\rm
Let $G$ be a Lie group. Then the inverse map $(\cdot)^{-1}: G\lon G$ is a difference operator and $(G, (\cdot)^{-1})$ is a difference Lie group.}
\end{ex}

\begin{ex}{\rm
Let $(G, \huaB)$ be a Rota-Baxter Lie group,  i.e. $G$ is a Lie group and $\huaB:G\to G$ is a smooth map satisfying
$$
\huaB(g)\huaB(h)=\huaB(g\Ad_{\huaB(g)}h),\quad \forall g,h\in G.
$$  If $\huaB $ is   invertible, then $(G, \huaB^{-1})$ is a difference Lie group.}
\end{ex}

\begin{ex}\label{ex:Gln}{\rm
Let $G$ be the real matrix Lie group $GL_n(\mathbb R)$. Then the adjugate map $(\cdot)^*: G\lon G$ is a difference operator and $(G, (\cdot)^*)$ is a difference Lie group.}
%Also, the map taking any $g\in G$ to $\overline{g}g^{-1}$ is a difference operator, where $\overline{g}$ is the complex conjugate of $g$.
\end{ex}

\begin{ex}\label{ex:Gln'}{\rm
Let $G$ be the complex matrix Lie group $GL_n(\mathbb C)$. Then $G$ endowed with the map taking any $g\in G$ to $\overline{g}g^{-1}$ is a difference Lie group, where $\overline{g}$ is the complex conjugate of $g$.}
\end{ex}

\begin{lem}
Let $(G, \huaD)$ be a difference Lie group. Then $\huaD(e_G)=e_G$, where $e_G$ is the unit  of $G$, and
\begin{equation}\label{defiinverse}
\huaD(g^{-1})=(\huaD(g)g)^{-1}g, \quad\forall g\in G.
\end{equation}
\end{lem}
\begin{proof}
By \eqref{defiD}, we have $\huaD(e_G)=e_G$. Since $e_G=\huaD(gg^{-1})=\huaD(g)g\huaD(g^{-1})g^{-1}$, it implies that $\huaD(g^{-1})=(\huaD(g)g)^{-1}g.$
\end{proof}

\begin{lem}\label{defirmkD+}
If $\huaD$ is a difference operator, then it induces a Lie group homomorphism $\huaD_{+}: G\lon G,$ defined by $\huaD_{+}(g)=\huaD(g)g$, for all $g\in G$.
\end{lem}
\begin{proof}
By $\eqref{defiD}$, we have $\huaD_{+}(gh)=\huaD(gh)gh=\huaD(g)g\huaD(h)h=\huaD_{+}(g)\huaD_{+}(h)$, Thus $\huaD_+: G\lon G$ is a Lie group homomorphism.
\end{proof}

\begin{defi}
Let $(G, \huaD)$ and $(G', \huaD')$ be two difference Lie groups. A {\bf homomorphism} from $(G, \huaD)$ to $(G', \huaD')$ consists of a Lie group homomorphism $\Psi :G\rightarrow G'$ such that
\begin{eqnarray}
\label{hom-rboG1}\huaD'\circ\Psi&=&\Psi\circ \huaD.
\end{eqnarray}
\end{defi}

\begin{defi}
A {\bf representation} of a difference Lie group $(G, \huaD)$ on a vector space $V$ with respect to a linear map $T: V\lon V$ is a Lie group homomorphism $\Theta$ from $G$ to $GL(V)$ such that
\begin{equation}\label{repofG}
T(\Theta(g)u)+\Theta(g)u=\Theta(\huaD(g)g)T(u)+\Theta(\huaD(g)g)u,\quad \forall g\in G, u\in V.
\end{equation}
\end{defi}
We denote a representation by $(V, T, \Theta)$.

\begin{ex}{\rm
For the difference Lie group $(GL_n(\mathbb R),(\cdot)^*)$ given in Example ~\ref{ex:Gln}, let $V=(\mathbb R^{n})^{\otimes n}$ and $(V,\Theta)$ be the tensor representation of $GL_n(\mathbb R)$. Let $\tilde{T}$ be the anti-symmetrization map on $V$ defined by
$$\tilde{T}(u_1\otimes\cdots\otimes u_n) =  \sum_{\epsilon\in S_n}(-1)^{|\epsilon|}u_{\epsilon(1)}\otimes\cdots\otimes u_{\epsilon(n)},\quad \forall u_1,\dots,u_n\in \mathbb R^{n}.$$
Denote $T=\tilde{T}-{\rm id}_V$. Then $(V,T,\Theta)$ is a representation of the difference Lie group $(GL_n(\mathbb R),(\cdot)^*)$. Indeed, for any $g\in GL_n(\mathbb R)$,
\begin{eqnarray*}
  \tilde{T}(\Theta(g)(u_1\otimes\cdots\otimes u_n)) &=&
  \tilde{T}(gu_1\otimes\cdots\otimes gu_n)\\
   &=&
  \sum_{\epsilon\in S_n}(-1)^{|\epsilon|}gu_{\epsilon(1)}\otimes\cdots\otimes gu_{\epsilon(n)}\\
   &=&
  \det(g)\sum_{\epsilon\in S_n}(-1)^{|\epsilon|}u_{\epsilon(1)}\otimes\cdots\otimes u_{\epsilon(n)}\\
  &=&
  \Theta(g^*g)\tilde{T}(u_1\otimes\cdots\otimes u_n).
\end{eqnarray*}
Hence, Eq.~\eqref{repofG} holds for $(V,T,\Theta)$.
}
\end{ex}

\begin{ex}{\rm
For the difference Lie group $(GL_n(\mathbb C),\overline{(\cdot)}(\cdot)^{-1})$ given in Example ~\ref{ex:Gln'}, let $V=\mathbb C^{n}$ and $(V,\Theta)$ be the vector representation of $GL_n(\mathbb C)$. Let $T$ be an $\mathbb R$-linear operator on $V$ defined by
$$T(u)=\overline{u}-u,\quad\forall u\in V.$$
Then $(V,T,\Theta)$ is a real representation of $(GL_n(\mathbb C),\overline{(\cdot)}(\cdot)^{-1})$. In fact, for any $g\in GL_n(\mathbb C)$, we have
\begin{eqnarray*}
T(\Theta(g)u)+\Theta(g)u = \overline{gu}= \Theta(\overline{g})\overline{u}= \Theta((\overline{g}g^{-1})g)(T(u)+u).
\end{eqnarray*}
Hence, Eq.~\eqref{repofG} holds for $(V,T,\Theta)$.
}
\end{ex}

Let $G$ be a Lie group and $\g$ be the corresponding Lie algebra of $G$. Since $\Ad(g)\in\Aut(G)$ for all $g\in G$ and $\Ad(g)e_G=e_G$, it follows that $\Ad(g)_{*e_G}: \g\rightarrow\g$ is an isomorphism of Lie algebras. By $\Ad(g_1 g_2)=\Ad(g_1)\Ad(g_2)$, we have $\Ad(g_1 g_2)_{*e_G}=\Ad(g_1)_{*e_G}\Ad(g_2)_{*e_G}$. Thus we obtain a Lie group homomorphism from the Lie group $G$ to $\Aut(\g)$, which is also denoted by $\Ad: G\lon \Aut(\g)$.

\begin{pro}
Let $(G, \huaD)$ be a difference Lie group. Then $(\g, D, \Ad)$ is a representation of $(G, \huaD)$, where $D=\huaD_{*e_G}$.
\end{pro}
\begin{proof}
For any $g\in G$ and $x\in\g$, since $\Ad: G\lon \Aut(\g)$ is a Lie group homomorphism and $\Ad(g)x=\frac{d}{dt}\big|_{t=0}g\exp(tx)g^{-1}$, by \eqref{defiinverse}, we have
\begin{eqnarray*}
&&D(\Ad(g)x)\\
&=&\frac{d}{dt}\bigg|_{t=0}\huaD(g\exp(tx)g^{-1})\\
&=&\frac{d}{dt}\bigg|_{t=0}\huaD(g\exp(tx))g\exp(tx)\huaD(g^{-1})\exp(-tx)g^{-1}\\
&=&\frac{d}{dt}\bigg|_{t=0}\huaD(g)g\huaD(\exp(tx))\exp(tx)\huaD(g^{-1})\exp(-tx)g^{-1}\\
&=&\frac{d}{dt}\bigg|_{t=0}\Big((\huaD(g)g)\huaD(\exp(tx))(\huaD(g)g)^{-1}\Big)\Big((\huaD(g)g)\exp(tx)(\huaD(g)g)^{-1}\Big)\Big(g\exp(-tx)g^{-1}\Big)\\
&=&\frac{d}{dt}\bigg|_{t=0}(\huaD(g)g)\huaD(\exp(tx))(\huaD(g)g)^{-1}+\frac{d}{dt}\bigg|_{t=0}(\huaD(g)g)\exp(tx)(\huaD(g)g)^{-1}+\frac{d}{dt}\bigg|_{t=0}g\exp(-tx)g^{-1}\\
&=&\Ad(\huaD(g)g)(D(x))+\Ad(\huaD(g)g)x-\Ad(g)x.
\end{eqnarray*}
Thus $(\g, D, \Ad)$ is a representation of the difference Lie group $(G, \huaD)$.
\end{proof}

\begin{thm}\label{th:semidirect}
Let $(V, T, \Theta)$ be a representation of a difference Lie group $(G, \huaD)$. Then $(G\ltimes_\Theta V, \huaD_{\ltimes})$ is a difference Lie group, where $G\ltimes_\Theta V$ is the semidirect product Lie group, in which the multiplication $\cdot_\ltimes$ is given by
\begin{equation*}
(g, u)\cdot_{\ltimes}(h, v)=(gh, u+\Theta(g)v), \quad \forall g, h\in G, u, v\in V,
\end{equation*}
and $\huaD_{\ltimes}$ is given by
\begin{equation*}
\huaD_{\ltimes}(g, u)=(\huaD(g), T(u)+u-\Theta(\huaD(g))u),\quad \forall g\in G, u\in V.
\end{equation*}
\end{thm}
\begin{proof}
Since $(V, T, \Theta)$ is a representation of the difference Lie group $(G, \huaD)$, by \eqref{defiD} and \eqref{repofG}, for all $g, h\in G$ and $u, v\in V$, we have
\begin{eqnarray*}
&&\huaD_{\ltimes}((g, u)\cdot_\ltimes(h, v))\cdot_\ltimes(g, u)\\
&=&\huaD_{\ltimes}(gh, u+\Theta(g)v)\cdot_\ltimes(g, u)\\
&=&\Big(\huaD(gh), T(u)+T(\Theta(g)v)+u+\Theta(g)v-\Theta(\huaD(gh))u-\Theta(\huaD(gh))\Theta(g)v\Big)\cdot_\ltimes(g, u)\\
&=&\Big(\huaD(gh)g, T(u)+T(\Theta(g)v)+u+\Theta(g)v-\cancel{\Theta(\huaD(gh))u}-\Theta(\huaD(gh))\Theta(g)v+\cancel{\Theta(\huaD(gh))u}\Big)\\
&=&\Big(\huaD(g)g\huaD(h), T(u)+T(\Theta(g)v)+u+\Theta(g)v-\Theta(\huaD(g)g\huaD(h))v\Big)\\
&=&\Big(\huaD(g)g\huaD(h), T(u)+u+\Theta(\huaD(g)g)T(v)+\Theta(\huaD(g)g)v-\Theta(\huaD(g)g\huaD(h))v\Big)
\end{eqnarray*}
and
\begin{eqnarray*}
&&\huaD_{\ltimes}(g, u)\cdot_\ltimes(g, u)\cdot_\ltimes\huaD_{\ltimes}(h, v)\\
&=&(\huaD(g), T(u)+u-\Theta(\huaD(g))u)\cdot_{\ltimes}(g, u)\cdot_{\ltimes}(\huaD(h), T(v)+v-\Theta(\huaD(h))v)\\
&=&(\huaD(g)g, T(u)+u-\cancel{\Theta(\huaD(g)u)}+\cancel{\Theta(\huaD(g))u})\cdot_{\ltimes}(\huaD(h), T(v)+v-\Theta(\huaD(h))v)\\
%&=&(\huaD(g)g\huaD(h), T(u)+u-\cancel{\Theta(\huaD(g)u)}+\cancel{\Theta(\huaD(g))u}+\Theta(\huaD(g)g)T(v)+\Theta(\huaD(g)g)v-\Theta(\huaD(g)g)\Theta(\huaD(h))v)\\
&=&\Big(\huaD(g)g\huaD(h), T(u)+u+\Theta(\huaD(g)g)T(v)+\Theta(\huaD(g)g)v-\Theta(\huaD(g)g\huaD(h))v\Big),
\end{eqnarray*}
which implies
\begin{equation*}
\huaD_{\ltimes}((g, u)\cdot_{\ltimes}(h, v))=\huaD_{\ltimes}(g, u)\cdot_{\ltimes}(g, u)\cdot_{\ltimes}\huaD_{\ltimes}(h, v)\cdot_{\ltimes}(g, u)^{-1}.
\end{equation*}
Thus $(G\ltimes_\Theta V, \huaD_{\ltimes})$ is a difference Lie group.
\end{proof}

\subsection{Differentiation and integration of representations}
In  this subsection, we establish the relationship between representations of difference Lie groups and representations of difference Lie algebras via differentiation and integration. Recall representations of difference Lie algebras as following.

\begin{defi}
A {\bf difference operator on a Lie algebra} $\g$ is a linear map $D:\g\lon\g$ such that
\begin{equation*}
D([x,y])=[D(x), y]+[x, D(y)]+[D(x), D(y)],\quad \forall x, y\in\g.
\end{equation*}
A difference Lie algebra $(\g, D)$ is a Lie algebra $\g$ equipped with a difference operator $D:\g\lon\g$.
\end{defi}

\begin{defi}\cite{JS}
A {\bf representation of a difference Lie algebra} $(\g, D)$ on a vector space $V$ with respect to a linear map $T: V\lon V$ is a Lie algebra representation $\theta:\g\to\gl(V)$ such that
\begin{equation*}
T(\theta(x)u)=\theta(D(x))u+\theta(x)T(u)+\theta(D(x))T(u), \quad \forall x\in\g, u\in V.
\end{equation*}
\end{defi}
Denote a representation by $(V, T, \theta)$.

It was proved in \cite{GLS} that the differentiation of a difference Lie group $(G, \huaD)$ is the difference Lie algebra $(\g, D)$, where $\g$ is the Lie algebra of the Lie group $G$ and $D=\huaD_{*e_G}$. A difference Lie algebra can also be integrated to a difference Lie group \cite{JS}.

\begin{thm}\label{thm-rep-DGL}
Let $(V, T, \Theta)$ be a representation of a difference Lie group $(G, \huaD)$. Then $(V, T, \theta)$ is a representation of the difference Lie algebra $(\g, D)$, where $\theta=\Theta_{*e_G}$.
\end{thm}
\begin{proof}
Since $T: V\lon V$ is a linear map, we have $T_{*}=T$. Denote the exponential map of the Lie group $G$ by $\exp$. For any $x\in\g, u\in V$,  by \eqref{repofG}, we have
\begin{eqnarray*}
T(\theta(x)u)&=&\frac{d}{dt}\bigg|_{t=0}T(\Theta(\exp(tx))u)\\
&=&\frac{d}{dt}\bigg|_{t=0}\Big(\Theta(\huaD(\exp(tx))\exp(tx))T(u)+\Theta(\huaD(\exp(tx))\exp(tx))u-\Theta(\exp(tx))u\Big)\\
&=&\frac{d}{dt}\bigg|_{t=0}\Theta(\huaD(\exp(tx)))T(u)+\frac{d}{dt}\bigg|_{t=0}\Theta(\exp(tx))T(u)+\frac{d}{dt}\bigg|_{t=0}\Theta(\huaD(\exp(tx)))u\\
&=&\theta(D(x))T(u)+\theta(x)T(u)+\theta(D(x))u.
\end{eqnarray*}
Thus $(V, T, \theta)$ is a representation of the difference Lie algebra $(\g, D)$.
\end{proof}

Let $(\g, D)$ be a difference Lie algebra and $(V, T, \theta)$ be a representation. We denote the integration of $(\g, D)$ by $(G, \huaD)$, where $G$ is a connected and simply connected Lie group. See \cite{JS} for explicit construction of $\huaD$. Let $\Theta: G\lon GL(V)$ be the Lie group homomorphism integrating the Lie algebra homomorphism $\theta: \g\lon\gl(V)$. Then we have the following theorem.

\begin{thm}
With the above notations, if $(V, T, \theta)$ is a representation of a difference Lie algebra $(\g, D)$, then $(V, T, \Theta)$ is a representation of the integrated difference Lie group $(G, \huaD)$.
\end{thm}
\begin{proof}
Since $\Theta$ is already a Lie group homomorphism, we only need to show that  \eqref{repofG} holds.

Define $D_+:\g\to\g$ by
$$
D_+(x)=x+D(x).
$$
Then $D_+$ is a Lie algebra homomorphism, and the graph of $D_+$, which is denoted by $
\mathrm{Gr}(D_+)=\{(x, D(x)+x)\,|\,\forall x\in\g\}
$
is a Lie subalgebra of the direct sum Lie algebra $\g\oplus\g$. Denote by $$
\widetilde{\gl}(V, T)=\{(\phi_1, \phi_2)\,|\,\phi_1, \phi_2\in\gl(V), \text{and}~~(T+\Id)\circ\phi_1=\phi_2\circ(T+\Id)\},
$$
which is a Lie subalgebra of the direct sum Lie algebra $\gl(V)\oplus \gl(V)$. Define $\Lambda:\mathrm{Gr}(D_+)\to \gl(V)\oplus \gl(V)$   by
$$\Lambda(x, D(x)+x)=(\theta(x), \theta(D(x))+\theta(x)), \quad\forall x\in\g.$$ Since $(V, T, \theta)$ is a representation of the difference Lie algebra $(\g, D)$, it follows that
$(\theta(x), \theta(D(x))+\theta(x))\in \widetilde{\gl}(V, T)$, and $\Lambda $ is a Lie algebra homomorphism from $\mathrm{Gr}(D_+)$ to $\widetilde{\gl}(V, T)$.

By Lemma \ref{defirmkD+}, $\huaD_+$ is a Lie group homomorphism. So the graph of  $\huaD_+$, which is denoted by $
\mathrm{Gr}(\huaD_+)=\{(g, \huaD(g)g)\,|\, \forall g\in G\},$ is a connected and simply connected Lie subgroup of the direct product Lie group $G\times G$. It is straightforward to see that the tangent map of $\huaD_+$ at the identity is exactly $D_+$, and the  Lie algebra of $\mathrm{Gr}(\huaD_+)$ is $\mathrm{Gr}(D_+)$. This can be summarized by the following commutative diagram:
 $$\xymatrix{
			G \ar[r]^(0.45){\huaD_+} &G\\
			\g \ar[u]^{\exp}\ar[r]_{D_+}& \g.\ar[u]^{\exp}
		}$$
Denote   by
\begin{equation*}
\widetilde{GL}(V, T)=\{(\Phi_1, \Phi_2)\,|\,\Phi_1, \Phi_2\in GL(V), \text{and}~~ (T+\Id)\circ\Phi_1=\Phi_2\circ(T+\Id)\},
\end{equation*}
which  is  obviously a Lie subgroup of the direct product Lie group $GL(V)\times GL(V)$, whose Lie algebra is $\widetilde{\gl}(V, T)$.

Let   $\Xi: \mathrm{Gr}(\huaD_+)\lon \widetilde{GL}(V, T)$ be the integration of the Lie algebra homomorphism $\Lambda:\mathrm{Gr}(D_+)\lon \widetilde{\gl}(V, T)$. Then we have
\begin{eqnarray*}
\Xi(\exp(x), \exp(D(x)+x))&=&(\EXP(\theta(x)), \EXP(\theta(D(x))+\theta(x)))\\
&=&(\Theta(\exp(x)), \Theta(\exp(D(x)+x))),
\end{eqnarray*}
where $\EXP$ is the exponential map from $\gl(V)$ to $GL(V)$.
Since $\exp(D(x)+x)=\huaD(\exp(x))\exp(x)$, we have
$$
\Xi(\exp(x), \huaD(\exp(x))\exp(x))=(\Theta(\exp(x)), \Theta(\huaD(\exp(x))\exp(x))).
$$
Since $\mathrm{Gr}(\huaD_+)$ is diffeomorphic to $G$ and $G$ is a connected Lie group, any $g\in G$ can be written as products of elements near the identity. Thus it follows that
$$
\Xi(g, \huaD(g)g)=(\Theta(g), \Theta(\huaD(g)g))\in \widetilde{GL}(V, T).
$$
Therefore, we have
$$
(T+\Id)\circ \Theta(g)= \Theta(\huaD(g)g))\circ(T+\Id),
$$
which implies that \eqref{repofG} holds, and
  $\Theta$ is a representation of $(G, \huaD)$ on $V$ with respect to $T$.
\end{proof}

\section{Cohomologies  of difference   Lie groups}\label{sec:coh}

In this section, we introduce cohomology theories for difference operators on Lie groups as well as difference Lie groups. The relation between various cohomologies are given by a long exact sequence.

 First we  recall the normalized cohomology of a Lie group $G$ with coefficients in a representation
 $\Theta:G\to GL(V)$ (see e.g. \cite{EM}). A smooth map $\alpha_n: \underbrace{G\times\cdots\times G}_n\lon V$ is called an {\bf $n$-normalized cochain} if $\alpha_n(g_1, \cdots, g_n)=0$ when any one of elements $g_i=e_G$. Denote the space of $n$-normalized cochains by $C^{n}(G, V)$, which is an abelian group. The coboundary operator $\dM^\Theta: C^{n}(G, V)\rightarrow C^{n+1}(G, V)$ is defined by
\begin{eqnarray*}
\dM^\Theta(\alpha_n)(g_1,\cdots,g_n,g_{n+1})&=&\Theta(g_1)\alpha_n(g_2, \cdots, g_n, g_{n+1})\\
&&+\sum_{i=1}^{n}(-1)^{i}\alpha_n(g_1,\cdots,g_{i-1}, g_ig_{i+1}, g_{i+2}, \cdots g_{n+1})\\
&&+(-1)^{n+1}\alpha_{n}(g_1,\cdots,g_n).
\end{eqnarray*}
The corresponding $n$-th cohomology group is denoted by $\huaH^{n}(G, V)$.

\begin{rmk}
It was proved in \cite{EM}  that the   normalized cohomology of a Lie group $G$ with coefficients in a representation $\Theta:G\to GL(V)$ is isomorphic to   the usual cohomology.
\end{rmk}

\subsection{Cohomologies  of difference operators on  Lie groups}
In this subsection, we define a cohomology theory for difference operators on Lie groups.

\begin{thm}\label{defiarepD}
 Let $(V, T, \Theta)$ be a representation of a difference Lie group $(G, \huaD)$. Define $\Theta_\huaD:G\to GL(V)$ by
 \begin{equation}\label{defiarepD1}
 \Theta_\huaD(g)u=\Theta(\huaD(g)g)u,\quad \forall g\in G, u\in V.
 \end{equation}
 Then $\Theta_\huaD$ is a representation of $G$ on $V$.
\end{thm}

\begin{proof}
  For all $g, h\in G, u\in V$, by the fact that $\Theta$ is a Lie group homomorphism and \eqref{defiD}, we have
\begin{eqnarray*}
\Theta_\huaD(gh)u=\Theta(\huaD(gh)gh)u=\Theta(\huaD(g)g\huaD(h)h)u=\Theta(\huaD(g)g)\Theta(\huaD(h)h)u=\Theta_\huaD(g)\Theta_\huaD(h)u,
\end{eqnarray*}
which implies that $\Theta_\huaD$ is a representation of $G$ on $V$.
\end{proof}

Now we are ready to define a cohomology theory for difference operators on Lie groups. Let $(V, T, \Theta)$ be a representation of a difference Lie group $(G, \huaD)$. Define the space of $1$-normalized cochains $\frkC^{1}(\huaD, T)$ by $0$. For $n\geq2$, define the space of $n$-normalized cochains $\frkC^n(\huaD, T)$ by $C^{n-1}(G, V)$.

\begin{defi}
The cohomology of the cochain complex $(\oplus_{n=1}^{\infty}\frkC^{n}( \huaD, T), \dTD)$ is called {\bf the cohomology of the difference operator} $\huaD$ with coefficients in the representation $(V, T, \Theta)$, where $\dTD$ is the coboundary operator for the Lie group $G$ with coefficients in the representation $(V; \Theta_\huaD)$. The corresponding $n$-th cohomology group is denoted by $\huaH^{n}(\huaD, T)$.
\end{defi}

\subsection{Cohomologies  of difference    Lie groups}

Let $(V, T, \Theta)$ be a representation of a difference Lie group $(G, \huaD)$. Define the space of $1$-cochains $C^{1}(G, \huaD, V, T)$ to be $C^1(G, V)$. For $n\geq 2$, we define the space of $n$-cochains $C^{n}(G, \huaD, V, T)$ by
\begin{equation*}
C^{n}(G, \huaD, V, T)=C^n(G, V)\oplus \frkC^{n}(\huaD, T).
\end{equation*}
Define the coboundary operator
\begin{equation*}
\delta: C^{n}(G, \huaD, V, T)\lon C^{n+1}(G, \huaD, V, T)
\end{equation*}
by
\begin{equation}\label{defidgco}
\delta(\alpha_n, \beta_{n-1})=(\dM^{\Theta}\alpha_n, \dTD(\beta_{n-1})+\frkK(\alpha_n) ),
\end{equation}
where $\dM^{\Theta}$ and $\dTD$ are the coboundary operators of the Lie group $G$ with coefficients in the representation $(V,\Theta)$ and $(V,\Theta_\huaD)$ respectively, and
  $\frkK: C^{n}(G, V)\lon \frkC^{n+1}(\huaD, T)$ is defined by $\frkK(\alpha_n)=\pk(\alpha_n)+\hk(\alpha_n)$, where
\begin{equation}\label{defisalpha}
\pk(\alpha_n)(g_1, \cdots, g_n)=
\begin{cases}
-\Theta(\huaD(g_1))\alpha_1(g_1)+\alpha_1(\huaD(g_1)g_1)-\alpha_1(\huaD(g_1)), & n=1;\\
\alpha_2(\huaD(g_1), g_1)-\alpha_2(\huaD(g_1g_2), g_1g_2)+\Theta(\huaD(g_1)g_1)\alpha_2(\huaD(g_2), g_2), & n=2;\\
0, & n\geq 3,
\end{cases}
\end{equation}
and
\begin{eqnarray}\label{defi-coht}
&&\hk(\alpha_n)(g_1,\cdots,g_n)\\
\nonumber&=&(-1)^{n}\Big(\alpha_n(\huaD(g_1)g_1,\cdots,\huaD(g_n)g_n)-T(\alpha_n(g_1,\cdots,g_n))-\alpha_n(g_1,\cdots,g_n)\Big).
\end{eqnarray}
Since $\alpha_n$ is an $n$-normalized cochain, it follows that $\frkK(\alpha_n)\in\frkC^{n+1}(\huaD, T)$.

\begin{thm}\label{th:coboundary}
With the above notations, $(\oplus_{n=1}^{\infty}C^{n}(G, \huaD, V, T), \delta)$ is a cochain complex, i.e.
\begin{equation*}
\delta\circ\delta=0.
\end{equation*}
\end{thm}
\begin{proof}
For any $(\alpha_n, \beta_{n-1})\in C^{n}(G, \huaD, V, T)$, by \eqref{defidgco}, we have
$$
\delta\circ\delta(\alpha_n, \beta_{n-1})=\Big(\dM^{\Theta}(\dM^{\Theta}\alpha_n), \dTD\Big(\dTD(\beta_{n-1})+\frkK(\alpha_{n})\Big)+\frkK(\dM^{\Theta}\alpha_{n})\Big).
$$
Since $\dM^{\Theta}$ and $\dTD$ are the coboundary operators of the Lie group $G$ with coefficients in the representations $(V,\Theta)$ and $(V,\Theta_\huaD)$ respectively, we get  $\dM^{\Theta}\circ\dM^{\Theta}=0$ and $\dTD\circ \dTD=0$.
Thus we only need to prove
$$\dTD\circ\frkK+\frkK\circ\dM^{\Theta}=0.$$

For $n\geq1$ and $g_1,\cdots,g_{n+1}\in G$, by \eqref{defiD}, \eqref{repofG} and \eqref{defi-coht}, we have
\begin{eqnarray*}
&&\dTD(\hk(\alpha_n))(g_1,\cdots, g_{n+1})\\
&=&\Theta(\huaD(g_1)g_1)\hk(\alpha_n)(g_2, \cdots, g_{n+1})\\
&&+\sum_{i=1}^{n}(-1)^i\hk(\alpha_n)(g_1,\cdots,g_ig_{i+1},\cdots,g_{n+1})+(-1)^{n+1}\hk(\alpha_{n})(g_1,\cdots,g_{n})\\
&=&(-1)^{n}\Theta(\huaD(g_1)g_1)\alpha_n(\huaD(g_2)g_2, \cdots, \huaD(g_{n+1})g_{n+1})\\
&&-(-1)^{n}\Theta(\huaD(g_1)g_1)T(\alpha_n(g_2, \cdots, g_{n+1}))-(-1)^{n}\Theta(\huaD(g_1)g_1)\alpha_{n}(g_2, \cdots, g_{n+1})\\
&&+\sum_{i=1}^{n}(-1)^{i+n}\Big(\alpha_{n}(\huaD(g_1)g_1, \cdots, \huaD(g_ig_{i+1})g_{i}g_{i+1},\cdots,\huaD(g_{n+1})g_{n+1})\\
&&-T(\alpha_{n}(g_1,\cdots,g_ig_{i+1},\cdots,g_{n+1}))-\alpha_n(g_1,\cdots,g_ig_{i+1},\cdots,g_{n+1})\Big)\\
&&-\alpha_n(\huaD(g_1)g_1,\cdots,\huaD(g_n)g_n)+T(\alpha_n(g_1,\cdots,g_n))+\alpha_n(g_1,\cdots,g_n)\\
&=&-\Big((-1)^{n+1}\Theta(\huaD(g_1)g_1)\alpha_{n}(\huaD(g_2)g_2,\cdots,\huaD(g_{n+1})g_{n+1})\\
&&+\sum_{i=1}^{n}(-1)^{i+n+1}\alpha_{n}(\huaD(g_1)g_1,\cdots,\huaD(g_i)g_i\huaD(g_{i+1})g_{i+1},\cdots,\huaD(g_{n+1})g_{n+1})\\
&&+\alpha_n(\huaD(g_1)g_1,\cdots,\huaD(g_n)g_n)-\alpha_n(g_1,\cdots,g_n)-\sum_{i=1}^n(-1)^{n+i+1}\alpha_{n}(g_1,\cdots,g_ig_{i+1},\cdots, g_{n+1})\\
&&+(-1)^{n}\Theta(g_1)\alpha_{n}(g_2, \cdots, g_{n+1})+(-1)^{n}T(\Theta(g_1)\alpha_{n}(g_2, \cdots, \g_{n+1}))\\
&&+\sum_{i=1}^{n}(-1)^{n+i}T(\alpha_{n}(g_1,\cdots,g_ig_{i+1},\cdots g_{n+1}))-T(\alpha_{n}(g_1, \cdots, g_n))\Big)\\
&=&(-1)^{n}\Big(\dM^{\Theta}\alpha_n(\huaD(g_1)g_1,\cdots, \huaD(g_{n+1})g_{n+1})-T(\dM^{\Theta}\alpha_{n}(g_1, \cdots, g_{n+1}))-\dM^{\Theta}\alpha_{n}(g_1,\cdots,g_{n+1})\Big)\\
&=&-\hk(\dM^{\Theta}(\alpha_n))(g_1,\cdots,g_{n+1}),
\end{eqnarray*}
which implies that $\dTD\circ \hk +\hk\circ\dM^{\Theta} =0$. Since when $n\geq3$, we have
$$
\dTD\circ\frkK +\frkK\circ\dM^{\Theta} =\dTD\circ\hk +\hk\circ\dM^{\Theta}.
$$
Thus $\dTD\circ\frkK+\frkK\circ\dM^{\Theta}=0$, when $n\geq 3$.

When $n=1$,
\begin{eqnarray*}
&&\Big(\dTD(\frkK(\alpha_1))+\frkK(\dM^{\Theta}\alpha_1)\Big)(g_1, g_2)\\
&=&\Big(\dTD(\pk(\alpha_1))+\pk(\dM^{\Theta}\alpha_1)\Big)(g_1, g_2)+\Big(\dTD(\hk(\alpha_1))+\hk(\dM^{\Theta}\alpha_1)\Big)(g_1, g_2)\\
&=&-\Theta(\huaD(g_1)g_1)\Theta(\huaD(g_2))\alpha_1(g_2)+\Theta(\huaD(g_1)g_1)\alpha_1(\huaD(g_2)g_2)-\Theta(\huaD(g_1)g_1)\alpha(\huaD(g_2))\\
&&+\Theta(\huaD(g_1g_2))\alpha_1(g_1g_2)-\alpha_1(\huaD(g_1g_2)g_1g_2)+\alpha_1(\huaD(g_1g_2))-\Theta(\huaD(g_1))\alpha_1(g_1)+\alpha_1(\huaD(g_1)g_1)\\
&&-\alpha_1(\huaD(g_1))+\Theta(\huaD(g_1))\alpha_1(g_1)-\alpha_1(\huaD(g_1)g_1)+\alpha_1(\huaD(g_1))-\Theta(\huaD(g_1g_2))\alpha_1(g_1g_2)\\
&&+\alpha_1(\huaD(g_1g_2)g_1g_2)-\alpha_1(\huaD(g_1g_2))+\Theta(\huaD(g_1)g_1)\Theta(\huaD(g_2))\alpha_1(g_2)-\Theta(\huaD(g_1)g_1)\alpha_1(\huaD(g_2)g_2)\\
&&+\Theta(\huaD(g_1)g_1)\alpha_1(\huaD(g_2))+\dTD(\hk(\alpha_1))(g_1, g_2)+\hk(\dM^{\Theta}\alpha_1)(g_1, g_2)\\
&=&0.
\end{eqnarray*}
When $n=2$, since $\dM^{\Theta}\alpha_2\in C^{3}(G, V)$, then $\frkK(\dM^{\Theta}\alpha_2)=\hk(\dM^{\Theta}\alpha_2)$. Thus it follows that
\begin{eqnarray*}
&&\Big(\dTD(\frkK(\alpha_2))+\frkK(\dM^{\Theta}\alpha_2)\Big)(g_1, g_2, g_3)\\
&=&\dTD(\pk(\alpha_2))(g_1, g_2, g_3)+\Big(\dTD(\hk(\alpha_2))+\hk(\dM^{\Theta}\alpha_2)\Big)(g_1, g_2, g_3)\\
&=&\Theta(\huaD(g_1)g_1)\pk(\alpha_2)(g_2, g_3)-\pk(\alpha_2)(g_1g_2, g_3)+\pk(\alpha_2)(g_1, g_2g_3)-\pk(\alpha_2)(g_1, g_2)\\
&=&\Theta(\huaD(g_1)g_1)\alpha_2(\huaD(g_2), g_2)-\Theta(\huaD(g_1)g_1)\alpha_2(\huaD(g_2g_3), g_2g_3)\\
&&+\Theta(\huaD(g_1)g_1)\Theta(\huaD(g_2)g_2)\alpha_2(\huaD(g_3), g_3)-\alpha_2(\huaD(g_1g_2), g_1g_2)+\alpha_2(\huaD(g_1g_2g_3), g_1g_2g_3)\\
&&-\Theta(\huaD(g_1g_2)g_1g_2)\alpha_2(\huaD(g_3), g_3)+\alpha_2(\huaD(g_1), g_1)-\alpha_2(\huaD(g_1g_2g_3), g_1g_2g_3)\\
&&+\Theta(\huaD(g_1)g_1)\alpha_2(\huaD(g_2g_3), g_2g_3)-\alpha_2(\huaD(g_1), g_1)\\
&&+\alpha_2(\huaD(g_1g_2), g_1g_2)-\Theta(\huaD(g_1)g_1)\alpha_2(\huaD(g_2), g_2)\\
&=&0.
\end{eqnarray*}
Thus for all $n\geq1$, $\delta\circ\delta=0$, which implies that $(\oplus_{n=1}^{\infty}C^{n}(G, \huaD, V, T), \delta)$ is a cochain complex.
\end{proof}

\begin{defi}
The cohomology of the cochain complex $(\oplus_{n=1}^{\infty}C^{n}(G, \huaD, V, T), \delta)$ is called {\bf the cohomology of the difference Lie group} with coefficients in the representation $(V, T, \Theta)$. The corresponding $n$-th cohomology group is denoted by $\huaH^{n}(G, \huaD, V, T)$.
\end{defi}

The relation between various cohomologies are given by the following theorem, which is resemblance of the Mayer-Vietoris sequence.

\begin{thm}\label{cohomology-exact-DG}
There is a short exact sequence of the  cochain complexes:
$$
0\longrightarrow(\oplus_{n=1}^{+\infty}\frkC^{n}(\huaD, T ),\dTD)\stackrel{\frak{i}}{\longrightarrow}(\oplus_{n=1}^{+\infty}C^n(G, \huaD, V, T), \delta)\stackrel{\frak{p}}{\longrightarrow} (\oplus_{n=1}^{+\infty}C^n(G, V),\dM^{\Theta})\longrightarrow 0,
$$
where $\frak{i}(\beta_{n-1})=(0,\beta_{n-1})$ and $\frak{p}(\alpha_{n},\beta_{n-1})=\alpha_{n}$ for all $\beta_{n-1}\in \frkC^{n}(\huaD, T)$ and $\alpha_{n}\in C^{n}(G, V)$.

Consequently,
there is a long exact sequence of the  cohomology groups:
$$
\cdots\longrightarrow\huaH^{n}(\huaD, T)\stackrel{\frak{i}_{*}}{\longrightarrow}\huaH^n(G,\huaD, V, T)\stackrel{\frak{p}_{*}}{\longrightarrow} \huaH^n(G, V)\stackrel{\frkk^n}\longrightarrow \huaH^{n+1}(\huaD, T)\longrightarrow\cdots,
$$
where the connecting map $\frkk^n$ is defined by
\begin{equation}\label{connmap}
\frkk^n([\alpha_n])=[\frkK(\alpha_n)],\quad  \forall [\alpha_n]\in \huaH^{n}(G, V).
\end{equation}
\end{thm}
\begin{proof}
By \eqref{defidgco}, we have the short exact sequence  of cochain complexes which induces a long exact sequence of cohomology groups. Moreover, if $\dM^{\Theta}\alpha_n=0$, then we can chose $(\alpha_n, 0)\in C^{n}(G, \huaD, V, T)$ such that $\frak{p}(\alpha_n, 0)=\alpha_n$. Since $\delta(\alpha_n, 0)=(\dM^{\Theta}\alpha_n, \frkK(\alpha_n))$, it follows that $\frkk^{n}([\alpha_n])=[\frkK(\alpha_n)]$.
\end{proof}

\section{The van Est theorem for cohomologies of difference Lie groups}\label{sec:van}
In this section, we establish the van Est theorem for cohomologies of difference Lie groups and cohomologies of difference Lie algebras, which can viewed as a justification of our
cohomology theory for difference Lie groups.
\subsection{Cohomologies of difference Lie algebras}
Let $(\g, D)$ be a difference Lie algebra. A cohomology theory of difference Lie algebras was introduced in \cite{JS} as follows. Let $(V, T, \theta)$ be a representation of $(\g, D)$.  Define the space of $1$-cochains $C^1(\g, D, V, T)$ to be $\Hom(\g,V)$. For $n\geq 2$, define the space of $n$-cochains $C^n(\g, D, V, T)$  by
$$
C^n(\g,D, V, T)=C^{n}(\g, V)\oplus \frkC^{n}(D, T),
$$
where $C^{n}(\g, V)=\Hom(\wedge^{n}\g,V)$, and $\frkC^{1}(D, T)=0, \frkC^{n}(D, T)=\Hom(\wedge^{n-1}\g ,V)$ for $n\geq2$.

Define the coboundary operator
$
\delta_{\theta}:C^n(\g,D, V, T)\lon C^{n+1}(\g, D, V, T)
$
by
\begin{equation}\label{deficodlg}
\delta_{\theta}(\zeta_n, \xi_{n-1})=(\dM^{\theta}\zeta_n, \dM^{\theta_{D}}\xi_{n-1}+K(\zeta_n)),
\end{equation}
for all $\zeta_n\in\Hom(\wedge^{n}\g,V), \xi_{n-1}\in\Hom(\wedge^{n-1}\g,V)$,
where $\dM^{\theta}$ and $\dM^{\theta_D}$ are the Chevalley-Eilenberg coboundary operators of the Lie algebra $\g$ with coefficients in the representation $(V, \theta)$ and $(V, \theta_D)$ respectively. Here $\theta_D: \g\lon\gl(V)$ is the representation of $\g$ on $V$ defined by $$\theta_D(x)u=\theta(x)u+\theta(D(x))u,$$
and $K: \Hom(\wedge^{n}\g,V)\rightarrow\Hom(\wedge^{n}\g, V)$ is defined by
\begin{eqnarray*}
&&K(\zeta_{n})(x_1,\cdots,x_{n})\\
\nonumber&=&(-1)^{n}\Big(\sum_{k=1}^{n}\sum_{1\leq i_1<\cdots<i_k\leq n}\zeta_n(x_1,\cdots,x_{i-1},D(x_{i_1}),\cdots,D(x_{i_k}),\cdots,x_n)-T(\zeta_n(x_1,\cdots,x_{n}))\Big).
\end{eqnarray*}
The corresponding $n$-th cohomology group is denoted by $H^{n}(\g, D, V, T)$. Moreover, we denote $n$-th cohomology groups of the cochain complexes $(\oplus_{n=1}^{+\infty}C^{n}(\g, V), \dM^{\theta})$ and $(\oplus_{n=1}^{\infty}\frkC^{n}(D, T), \dM^{\theta_D})$ by $H^{n}(\g, V)$ and $H^{n}(D, T)$ respectively. Then there is the following theorem.

\begin{thm}
With the above notations, there is a short exact sequence of the  cochain complexes:
$$
0\longrightarrow(\oplus_{n=1}^{+\infty}\frkC^{n}(D, T),\dM^{\theta_D})\stackrel{\iota}{\longrightarrow}(\oplus_{n=1}^{+\infty}C^n(\g, D, V, T), \delta_\theta)\stackrel{p}{\longrightarrow} (\oplus_{n=1}^{+\infty}C^{n}(\g, V),\dM^{\theta})\longrightarrow 0,
$$
where $\iota(\xi_{n-1})=(0,\xi_{n-1})$ and $p(\zeta_{n},\xi_{n-1})=\zeta_{n}$ for all $\xi_{n-1}\in \Hom(\wedge^{n-1}\g, V)$ and $\zeta_{n}\in \Hom(\wedge^{n}\g, V)$.

Consequently,
there is a long exact sequence of the  cohomology groups:
$$
\cdots\longrightarrow H^{n}(D, T)\stackrel{\iota_{*}}{\longrightarrow} H^n(\g, D, V, T)\stackrel{p_{*}}{\longrightarrow} H^n(\g, V)\stackrel{k^n}\longrightarrow H^{n+1}(D, T)\longrightarrow\cdots,
$$
where the connecting map $k^n$ is defined by
\begin{equation}\label{tranco}
k^n([\zeta_n])=[K(\zeta_n)],  \quad \forall [\zeta_n]\in H^{n}(\g, V).
\end{equation}
\end{thm}
\begin{proof}
By \eqref{deficodlg}, we have the short exact sequence of cochain complexes which induces a long exact sequence of cohomology groups. Moreover, if $\dM^{\theta}\zeta_n=0$, then we can chose $(\zeta_n, 0)\in C^{n}(\g, D, V, T)$ such that $p(\zeta_n, 0)=0$, which implies that $\delta_{\theta}(\zeta_n, 0)=(\dM^{\theta}\zeta_n, K(\zeta_n))$. Thus $k^{n}([\zeta_n])=[K(\zeta_n)]$.
\end{proof}

\subsection{The van Est theorem}

Let $G$ be a Lie group and $\g$ its Lie algebra. Let
 $\Theta:G\to GL(V)$ be a representation of $G$ and $\theta:\g\to\gl(V)$ the induced representation of $\g$.
Define $$\mathrm{VE}_n: C^{n}(G, V)\lon C^{n}(\g, V)$$ by
\begin{eqnarray*}
&&\mathrm{VE}_n(\alpha_n)(x_1, \cdots, x_n)\\
&=&\sum_{\epsilon\in S_{n}}(-1)^{|\epsilon|}\frac{d}{dt_{\epsilon(1)}}\cdots\frac{d}{dt_{\epsilon(n)}}\bigg|_{t_{\epsilon(1)}=\cdots=t_{\epsilon(n)}=0}\alpha_{n}\Big(\exp(t_{\epsilon(1)}x_{\epsilon(1)}), \cdots, \exp(t_{\epsilon(n)}x_{\epsilon(n)})\Big), \quad \forall x_1,\cdots, x_n\in\g.
\end{eqnarray*}
From the classical argument for the cohomologies of Lie groups and Lie algebras, $$\mathrm{VE}: \oplus_{n=1}^{\infty} C^{n}(G, V)\lon\oplus_{n=1}^{\infty}C^{n}(\g, V)$$ is a cochain map, which induces   homomorphisms ${\mathrm{VE}_{n}}_{*}$ from the cohomology group $\huaH^{n}(G, V)$ to $H^{n}(\g, V)$. Moreover, under certain conditions, the cohomology group $\huaH^{k}(G, V)$ and $H^{k}(\g, V)$ are isomorphic.

\begin{thm}\label{van}\rm(\cite{Van})
Let $G$ be a connected Lie group and its homotopy groups are trivial in $1, \cdots, n$, then for all $1\leq i\leq n$, the cohomology group $\huaH^{i}(G, V)$ is isomorphic to the cohomology group $H^{i}(\g, V)$.
\end{thm}

Let $(G, \huaD)$ be a difference Lie group and $(\g, D)$ be the corresponding difference Lie algebra. Let $(V, T, \Theta)$ be a representation of $(G, \huaD)$. By Theorem \ref{thm-rep-DGL}, $(V, T, \theta)$ is a representation of $(\g, D)$. By Theorem \ref{defiarepD}, $\Theta_\huaD$ defined by $\Theta_{\huaD}(g)u=\Theta(\huaD(g)g)u$ is a representation of $G$ on $V$. Moreover, we have
$$(\Theta_\huaD)_{* e_G}(x)u=\theta(D(x))u+\theta(x)u=\theta_D(x)u, \quad \forall x\in\g, u\in V, $$
which is a representation of $\g$ on $V$.

Define $\widetilde{\mathrm{VE}}_{n}: C^{n}(G, \huaD, V, T)\lon C^{n}(\g, D, V, T)$ by
$$
\widetilde{\mathrm{VE}}_{n}(\alpha_n, \beta_{n-1})=(\mathrm{VE}_{n}(\alpha_n), \mathrm{VE}_{n-1}(\beta_{n-1})),\quad \forall (\alpha_n, \beta_{n-1})\in C^{n}(G, \huaD, V, T).
$$
Then we have the following theorem.

\begin{thm}\label{chainhomtopy}
With the above notations, $\widetilde{\mathrm{VE}}: \oplus_{n=1}^{+\infty}C^{n}(G, \huaD, V, T)\lon \oplus_{n=1}^{+\infty}C^{n}(\g, D, V, T)$ is a cochain map, which induces homomorphisms $\widetilde{\mathrm{VE}}_{n*}$ from the cohomology group $\huaH^{n}(G, \huaD, V, T)$ to $H^{n}(\g, D, V, T)$. The map $\widetilde{\mathrm{VE}}$ is called the {\bf van Est map}.
\end{thm}
\begin{proof}
For $n\geq1, \alpha_n\in C^{n}(G, V),  \beta_{n-1}\in\frkC^{n}(\huaD, T)$, from the classical argument for the cohomologies of Lie groups and Lie algebras, we have
\begin{eqnarray*}
\widetilde{\mathrm{VE}}_{n+1}(\dM^{\Theta}\alpha_n, \frkK(\alpha_n)+\dTD\beta_{n-1})&=&(\mathrm{VE}_{n+1}(\dM^{\Theta}\alpha_n), \mathrm{VE}_{n}(\frkK(\alpha_n))+\mathrm{VE}_{n}(\dTD\beta_{n-1}))\\
&=&\Big(\dM^{\theta}(\mathrm{VE}_{n}(\alpha_n)), \mathrm{VE}_{n}(\frkK(\alpha_n))+\dM^{\theta_{D}}(\mathrm{VE}_{n-1}(\beta_{n-1}))\Big).
\end{eqnarray*}
Moreover, since $\frkK(\alpha_n)=\pk(\alpha_n)+\hk(\alpha_n)$ and denote $\mathrm{VE}_{n}(\alpha_n)$ by $\zeta_n$, by \eqref{defi-coht}, for any $x_1, \cdots, x_n\in\g$, we have
\begin{eqnarray*}
&&\mathrm{VE}_{n}(\hk(\alpha_n))(x_1, \cdots, x_n)\\
&=&(-1)^{n}\sum_{\epsilon\in S_n}(-1)^{|\epsilon|}\frac{d}{dt_{\epsilon(1)}}\cdots\frac{d}{dt_{\epsilon(n)}}\bigg|_{t_{\epsilon(1)}=\cdots=t_{\epsilon(n)}=0}\Big(-\alpha_{n}(\exp(t_{\epsilon(1)}x_{\epsilon(1)}), \cdots, \exp(t_{\epsilon(n)}x_{\epsilon(n)}))\\
&&+\alpha_{n}\Big(\huaD(\exp(t_{\epsilon(1)}x_{\epsilon(1)}))\exp(t_{\epsilon(1)}x_{\epsilon(1)}), \cdots, \huaD(\exp(t_{\epsilon(n)}x_{\epsilon(n)}))\exp(t_{\epsilon(n)}x_{\epsilon(n)})\Big)\\
&&-T(\alpha_{n}(\exp(t_{\epsilon(1)}x_{\epsilon(1)}), \cdots, \exp(t_{\epsilon(n)}x_{\epsilon(n)})))\Big)\\
&=&(-1)^{n}\sum_{\epsilon\in S_n}(-1)^{|\epsilon|}\frac{d}{dt_{\epsilon(1)}}\cdots\frac{d}{dt_{\epsilon(n)}}\bigg|_{t_{\epsilon(1)}=\cdots=t_{\epsilon(n)}=0}\Big(-\alpha_{n}(\exp(t_{\epsilon(1)}x_{\epsilon(1)}), \cdots, \exp(t_{\epsilon(n)}x_{\epsilon(n)}))\\
&&+\alpha_{n}\Big(\exp(t_{\epsilon(1)}(D(x_{\epsilon(1)})+x_{\epsilon(1)})), \cdots, \exp(t_{\epsilon(n)}(D(x_{\epsilon(n)})+x_{\epsilon(n)}))\Big)\\
&&-T(\alpha_{n}(\exp(t_{\epsilon(1)}x_{\epsilon(1)}), \cdots, \exp(t_{\epsilon(n)}x_{\epsilon(n)})))\Big)\\
&=&(-1)^{n}\Big(\zeta_n(D(x_1)+x_1, \cdots, D(x_n)+x_n)-T(\zeta_n(x_1, \cdots, x_n))-\zeta_n(x_1, \cdots, x_n)\Big)\\
&=&(-1)^{n}\Big(\sum_{k=1}^{n}\sum_{1\leq i_1<\cdots<i_k\leq n}\zeta_n(x_1,\cdots,x_{i-1},D(x_{i_1}),\cdots,D(x_{i_k}),\cdots,x_n)-T(\zeta_n(x_1,\cdots,x_{n}))\Big)\\
&=&K(\mathrm{VE}_{n}(\alpha_n))(x_1, \cdots, x_n).
\end{eqnarray*}
For $n=1$, denote $\mathrm{VE}_1(\alpha_1)$ by $\zeta_1$. By \eqref{defisalpha} and the fact that $\alpha_1$ is a $1$-normalized cochain, for any $x\in\g$, we have
\begin{eqnarray*}
\mathrm{VE}_{1}(\pk(\alpha_1))(x)&=&\frac{d}{dt}\bigg|_{t=0}\pk(\alpha_1)(\exp(tx))\\
&=&\frac{d}{dt}\bigg|_{t=0}\Big(\alpha_{1}(\huaD(\exp(tx))\exp(tx))-\alpha_1(\huaD(\exp(tx)))-\Theta(\huaD(\exp(tx)))\alpha_{1}(\exp(tx))\Big)\\
&=&\zeta_1(D(x))+\zeta_1(x)-\zeta_1(D(x))-\frac{d}{dt}\bigg|_{t=0}\Theta(\huaD(\exp(tx)))\alpha_{1}(e_G)-\frac{d}{dt}\bigg|_{t=0}\alpha_{1}(\exp(tx))\\
&=&\zeta_1(D(x))+\zeta_1(x)-\zeta_1(D(x))-\theta(D(x))\alpha(e_G)-\zeta_1(x)\\
&=&0.
\end{eqnarray*}
For $n=2$, denote $\mathrm{VE}_2(\alpha_2)$ by $\zeta_2$. By \eqref{defisalpha} and the fact that $\alpha_2$ is a $2$-normalized cochain, for any $x_1, x_2\in\g$, we have
\begin{eqnarray*}
&&\mathrm{VE}_{2}(\pk(\alpha_2))(x_1, x_2)\\
&=&\frac{d}{dt_1}\frac{d}{dt_2}\bigg|_{t_1=t_2=0}\pk(\alpha_2)(\exp(t_1x_1), \exp(t_2x_2))-\frac{d}{dt_2}\frac{d}{dt_1}\bigg|_{t_1=t_2=0}\pk(\alpha_2)(\exp(t_2x_2), \exp(t_1x_1))\\
&=&\frac{d}{dt_2}\frac{d}{dt_1}\bigg|_{t_1=t_2=0}\Big(\alpha_2(\huaD(\exp(t_1x_1)), \exp(t_1x_1))-\alpha_2(\huaD(\exp(t_1x_1)\exp(t_2x_2)), \exp(t_1x_1)\exp(t_2x_2))\\
&&+\Theta(\huaD(\exp(t_1x_1))\exp(t_1x_1))\alpha_2(\huaD(\exp(t_2x_2)), \exp(t_2x_2))\Big)\\
&&-\frac{d}{dt_1}\frac{d}{dt_2}\bigg|_{t_1=t_2=0}\Big(\alpha_2(\huaD(\exp(t_2x_2)), \exp(t_2x_2))-\alpha_2(\huaD(\exp(t_2x_2)\exp(t_1x_1)), \exp(t_2x_2)\exp(t_1x_1))\\
&&+\Theta(\huaD(\exp(t_2x_2)\exp(t_2x_2))\alpha_2(\huaD(\exp(t_1x_1)), \exp(t_1x_1))\Big)\\
&=&-\frac{d}{dt_2}\frac{d}{dt_1}\bigg|_{t_1=t_2=0}\alpha_2(\huaD(\exp(t_2x_2)), \exp(t_1x_1)\exp(t_2x_2))\\
&&-\frac{d}{dt_2}\frac{d}{dt_1}\bigg|_{t_1=t_2=0}\alpha_2(\huaD(\exp(t_1x_1)\exp(t_2x_2)), \exp(t_2x_2))\\
&&+\frac{d}{dt_2}\bigg|_{t_2=0}\theta(D(x_1)+x_1)\alpha_2(\huaD(\exp(t_2x_2)), \exp(t_2x_2))\\
&&+\frac{d}{dt_1}\frac{d}{dt_2}\bigg|_{t_1=t_2=0}\alpha_2(\huaD(\exp(t_1x_1)), \exp(t_2x_2)\exp(t_1x_1))\\
&&+\frac{d}{dt_1}\frac{d}{dt_2}\bigg|_{t_1=t_2=0}\alpha_2(\huaD(\exp(t_2x_2)\exp(t_1x_1)), \exp(t_1x_1))\\
&&-\frac{d}{dt_1}\bigg|_{t_1=0}\theta(D(x_2)+x_2)\alpha_2(\huaD(\exp(t_1x_1)), \exp(t_1x_1))\\
&=&-\frac{d}{dt_1}\frac{d}{dt_2}\bigg|_{t_1=t_2=0}(\alpha_2(e_G, \exp(t_1x_1)\exp(t_2x_2))+\alpha_2(\huaD(\exp(t_2x_2)), \exp(t_1x_1)))\\
&&-\frac{d}{dt_1}\frac{d}{dt_2}\bigg|_{t_1=t_2=0}(\alpha_2(\huaD(\exp(t_1x_1)), \exp(t_2x_2))+\alpha_2(\huaD(\exp(t_1x_1))\exp(t_2x_2), e_G))\\
&&+\frac{d}{dt_2}\bigg|_{t_2=0}\theta(D(x_1)+x_1)(\alpha_2(e_G, \exp(t_2x_2))+\alpha_2(\huaD(\exp(t_2x_2)), e_G))\\
&&+\frac{d}{dt_2}\frac{d}{dt_1}\bigg|_{t_1=t_2=0}(\alpha_2(\huaD(\exp(t_1x_1)), \exp(t_2x_2))+\alpha_2(e_G, \exp(t_2x_2)\exp(t_1x_1)))\\
&&+\frac{d}{dt_2}\frac{d}{dt_1}\bigg|_{t_1=t_2=0}(\alpha_2(\huaD(\exp(t_2x_2)), \exp(t_1x_1))+\alpha_2(\huaD(\exp(t_2x_2)\exp(t_1x_1)), e_G))\\
&&-\frac{d}{dt_1}\bigg|_{t_1=0}\theta(D(x_2)+x_2)(\alpha_2(e_G, \exp(t_1x_1))+\alpha_2(\huaD(\exp(t_1x_1)), e_G))\\
&=&0.
\end{eqnarray*}
Thus, $\mathrm{VE}_n(\frkK(\alpha_n))=\mathrm{VE}_n(\hk(\alpha_n))=K(\mathrm{VE}_n(\alpha_n))$, which implies that $\widetilde{\mathrm{VE}}$ is a cochain map. Therefore $\widetilde{\mathrm{VE}}_{n*}$ are homomorphisms from the cohomology group $\huaH^{n}(G, \huaD, V, T)$ to $H^{n}(\g, D, V, T)$.
%\begin{equation*}
%\tilde{\mathrm{VE}}_{n+1}(\dM^{\Theta}\alpha_n, \frkK(\alpha_n)+\dTD\beta_{n-1})=(\dM^{\theta}(\mathrm{VE}_{n}(\alpha_n)), \dM^{\theta_{D}}(\mathrm{VE}_{n-1}(\beta_{n-1}))+K(\mathrm{VE}_{n}(\alpha_n))).
%\end{equation*}
\end{proof}

\begin{thm}
 Assume that $G$ is a connected Lie group and its homotopy groups are trivial in $1, \cdots, n$. Then for $1\leq i\leq n$, the cohomology group $\huaH^{i}(G, \huaD, V, T)$ is isomorphic to the cohomology group $H^{i}(\g, D, V, T)$.
\end{thm}
\begin{proof}
For $[\alpha_{i-1}]\in\huaH^{i-1}(G, V), [\beta_{i-1}]\in\huaH^{i}(\huaD, T)$ and $[(\alpha_i, \beta'_{i-1})]\in\huaH^{i}(G, \huaD, V, T)$, by Theorem \ref{chainhomtopy}, we have
\begin{eqnarray*}
k^{i-1}(\mathrm{VE}_{i-1*}([\alpha_{i-1}]))&=&[K(\mathrm{VE}_{i-1}(\alpha_{i-1}))]=\mathrm{VE}_{i-1*}(\frkk^{i-1}([\alpha_{i-1}])),\\
\widetilde{\mathrm{VE}}_{i*}(\frak{i}_*(\beta_{i-1}))&=&[(0, \mathrm{VE}_{i-1}(\beta_{i-1}))]=\iota_*(\mathrm{VE}_{i-1*}([\beta_{i-1}])),\\
\mathrm{VE}_{i*}(\frak{p}_*([(\alpha_i, \beta'_{i-1})]))&=&[\mathrm{VE}_i(\alpha_{i})]=p_*(\widetilde{\mathrm{VE}}_{i*}([(\alpha_{i}, \beta'_{i-1})])),
\end{eqnarray*}
where $k^{i-1}$ and $\frkk^{i-1}$ are given by \eqref{tranco} and \eqref{connmap} respectively.  Thus we have the following commutative diagram:
\[\begin{CD}
\huaH^{i-1}(G, V)@>\frkk^{i-1}>>\huaH^{i}(\huaD, T)@>\frak{i}_*>>\huaH^{i}(G, \huaD, V, T)@>\frak{p}_*>>\huaH^{i}(G, V)            @>\frkk^{i}>>\huaH^{i+1}(\huaD, T)\\
@V \mathrm{VE}_{i-1*} VV    @V \mathrm{VE}_{i-1*} VV   @V\widetilde{\mathrm{VE}}_{i*}VV  @V \mathrm{VE}_{i*} VV    @V \mathrm{VE}_{i+1*} VV\\
H^{i-1}(\g, V)@>k^{i-1}>>H^{i}(D, T) @>\iota_*>>H^{i}(\g, D, V, T)@>p_*>>H^{i}(\g, V)             @>k^{i}>>H^{i+1}(D, T)
.
\end{CD}\]
Since $G$ is connected and its homotopy groups are trivial in $1, \cdots, n$, for $1\leq i\leq n$, by Theorem \ref{van}, we have the following group isomorphism $\huaH^{i}(\huaD, T)\simeq H^{i}(D, T), \huaH^{i+1}(\huaD, T)\simeq H^{i+1}(D, T)$ and $\huaH^{i-1}(G, V)\simeq H^{i-1}(\g, V), \huaH^{i}(G, V)\simeq H^{i}(\g, V)$. Apply the Five Lemma to the above diagram, we have
$$\huaH^{i}(G, \huaD, V, T)\simeq H^{i}(\g, D, V, T), \quad \forall 1\leq i\leq n.$$
Thus, for $1\leq i\leq n$, the cohomology group $\huaH^{i}(G, \huaD, V, T)$ is isomorphic to the cohomology group $H^{i}(\g, D, V, T)$.
\end{proof}

\section{Abelian extensions of difference Lie groups}\label{sec:ext}

In this section we use the established cohomology theory to study abelian extensions of difference Lie groups, and show that abelian extensions of difference Lie groups relative to a fix representation are classified by the second cohomology group.

\begin{defi}
Let $(G, \huaD)$ be a difference Lie group and $(V, T)$ be a vector space with a linear map $T: V\lon V$. An {\bf abelian extension} of $(G, \huaD)$ by $(V, T)$ is a short exact sequence of difference Lie group homomorphisms:
\[\begin{CD}
\{1\}@>>>V@>i>>\Pi @>p>>G            @>>>\{1\}\\
@.    @V T VV   @V\huaD_\Pi VV  @V \huaD VV    @.\\
\{1\}@>>>V @>i>>\Pi @>p >>G             @>>>\{1\}
,
\end{CD}\]
where $(\Pi, \huaD_{\Pi})$ is a difference Lie group.
\end{defi}

\begin{defi}
A {\bf section} of an abelian extension $(\Pi, \huaD_\Pi)$ of a difference Lie group $(G, \huaD)$ by $(V, T)$ is a smooth map $s: G\lon \Pi$ such that
\begin{equation*}
p\circ s=\Id, \quad s(e_G)=e_\Pi.
\end{equation*}
\end{defi}

Let $s$ be a section. Define a smooth map $\Theta: G\lon GL(V)$ by
\begin{equation*}
\Theta(g)u=s(g)\cdot_\Pi u\cdot_\Pi (s(g))^{-1}, \quad \forall g\in G, u\in V.
\end{equation*}
Then we have the following proposition.
\begin{pro}\label{pro-repn}
With the above notations, $\Theta: G\lon GL(V)$ is a representation of the difference Lie group $(G, \huaD)$ on $V$ with respect to the linear map $T$. Moreover, this representation is independent on the choice of sections.
\end{pro}
\begin{proof}
For any $g, h\in G$ and $u\in V$, since $V$ is a vector space, we have
\begin{eqnarray*}
\Theta(g\cdot_Gh)u&=&s(g\cdot_Gh)\cdot_\Pi u\cdot_\Pi (s(g\cdot_Gh))^{-1}\\
&=&s(g\cdot_Gh)\cdot_\Pi(s(g)\cdot_\Pi s(h))^{-1}+s(g)\cdot_\Pi s(h)\cdot_\Pi u\cdot_\Pi (s(h))^{-1}\cdot_\Pi (s(g))^{-1}\\
&&-(s(g\cdot_Gh)\cdot_\Pi(s(g)\cdot_\Pi s(h))^{-1})\\
&=&s(g)\cdot_\Pi s(h)\cdot_\Pi u\cdot_\Pi (s(h))^{-1}\cdot_\Pi (s(g))^{-1}\\
&=&\Theta(g)\Theta(h)u.
\end{eqnarray*}
Thus, $\Theta$ is a representation of the Lie group $G$ on the vector space $V$.

 Since $V$ is a vector space which is an abelian Lie group, we have $a\cdot_\Pi u\cdot_\Pi a^{-1}=b\cdot_\Pi u\cdot_\Pi b^{-1}$, where   $a, b\in\Pi$ such that $p(a)=p(b)$ and $u\in V$. For any $g\in G$, by the fact that $p(\huaD_\Pi(s(g))\cdot_\Pi s(g))=\huaD(g)\cdot_G g=p(s(\huaD(g)\cdot_G g))$ and $\huaD((s(g))^{-1})=(s(g))^{-1}\cdot_\Pi(\huaD(s(g)))^{-1}\cdot_\Pi s(g)$, we have
\begin{eqnarray*}
&&T(\Theta(g)u)+\Theta(g)u\\
&=&\huaD_{\Pi}(s(g)\cdot_\Pi u\cdot_\Pi (s(g)^{-1}))+s(g)\cdot_\Pi u\cdot_\Pi (s(g))^{-1}\\
&=&\huaD_{\Pi}(s(g))\cdot_\Pi s(g)\cdot_\Pi\huaD_\Pi(u)\cdot_\Pi u\cdot_\Pi\huaD((s(g))^{-1})\cdot_\Pi u^{-1}\cdot_\Pi (s(g))^{-1}+s(g)\cdot_\Pi u\cdot_\Pi (s(g))^{-1}\\
&=&(\huaD_{\Pi}(s(g))\cdot_\Pi s(g))\cdot_\Pi\huaD_\Pi(u)\cdot_\Pi(\huaD_{\Pi}(s(g))\cdot_\Pi s(g))^{-1}\\
&&+(\huaD_{\Pi}(s(g))\cdot_\Pi s(g))\cdot_\Pi u\cdot_\Pi(\huaD_{\Pi}(s(g))\cdot_\Pi s(g))^{-1}-s(g)\cdot_\Pi u\cdot_\Pi (s(g))^{-1}+s(g)\cdot_\Pi u\cdot_\Pi (s(g))^{-1}\\
&=&s(\huaD(g)\cdot_G g)\cdot_\Pi\huaD_\Pi(u)\cdot_\Pi(s(\huaD(g)\cdot_G g))^{-1}+s(\huaD(g)\cdot_G g)\cdot_\Pi u\cdot_\Pi(s(\huaD(g)\cdot_G g))^{-1}\\
&=&\Theta(\huaD(g)\cdot_G g)T(u)+\Theta(\huaD(g)\cdot_G g)u.
\end{eqnarray*}
Thus, $(V, T, \Theta)$ is a representation of the difference Lie group $(G, \huaD)$.

Let $s'$ be another section and $\Theta'$ be the corresponding representation of the difference Lie group $(G, \huaD)$. Since $(s'(g))^{-1}\cdot_\Pi s(g)\in V$, it follows  that $$((s'(g))^{-1}\cdot_\Pi s(g))\cdot_\Pi u\cdot_\Pi ((s'(g))^{-1}\cdot_\Pi s(g))^{-1}=u.$$ Thus, the representation $\Theta$ is independent on the choice of sections.
\end{proof}

The above result tells us that any abelian extension   $(\Pi, \huaD_\Pi)$ of a difference Lie group $(G, \huaD)$ by $(V, T)$ determines a representation $\Theta$ of the difference Lie group $(G, \huaD)$ on $V$ with respect to $T$. We will say that the abelian extension $(\Pi, \huaD_\Pi)$ is {\bf relative to the representation} $(V, T, \Theta)$ of the difference Lie group $(G, \huaD)$.

Let $s$ be a section. Define $\alpha\in C^{2}(G, V)$ and $ \beta\in\frkC^{2}(\huaD, T)$ by
\begin{eqnarray}
\label{2-cocycle}\alpha(g, h)&=&s(g)\cdot_\Pi s(h)\cdot_\Pi (s(g\cdot_G h))^{-1},\\
\label{2-cocycle'}\beta(g)&=&\huaD_\Pi(s(g))\cdot_\Pi (s(\huaD(g)))^{-1}.
\end{eqnarray}
Define $S: G\times V\lon \Pi$ by
\begin{equation*}
S(g, u)=u\cdot_\Pi s(g).
\end{equation*}
It is obvious that $S$ is an isomorphism between manifolds. Transfer the difference Lie group structure on $\Pi$ to $G\times V$ via the isomorphism $S$, we obtain a difference Lie group $(G\times V, \cdot_{\alpha}, \huaD_{\beta})$, where $\cdot_{\alpha}$ and $\huaD_{\beta}$ are given by
\begin{eqnarray*}
(g, u)\cdot_{\alpha}(h, v)&=&S^{-1}(u\cdot_\Pi s(g)\cdot_\Pi v\cdot_\Pi s(h))=(g\cdot_G h, u+\Theta(g)v+\alpha(g, h)),\\
\huaD_{\beta}(g, u)&=&S^{-1}(\huaD_{\Pi}(u\cdot_\Pi s(g)))=(\huaD(g), T(u)+u-\Theta(\huaD(g))u+\beta(g)).
\end{eqnarray*}

\begin{thm}\label{thmabeH2}
With the above notations, $(\alpha, \beta)$ is a $2$-cocycle of the difference Lie group $(G, \huaD)$ with coefficients in the representation $(V, T, \Theta)$. Moreover, its cohomological class does not depend on the choice of sections.
\end{thm}
\begin{proof}
By the fact that $\cdot_{\alpha}$ is a group multiplication, we deduce that $\alpha$ is $2$-cocycle of the Lie group $(G, \cdot_G)$ with coefficients in $(V, \Theta)$, i.e. $\dM^{\Theta}\alpha=0$.

Moreover, we have
\begin{eqnarray*}
&&\huaD_{\beta}\Big((g, u)\cdot_{\alpha} (h, v)\Big)\cdot_{\alpha}\Big((g, u)\cdot_{\alpha}(h, v)\Big)\\
&=&\huaD_{\beta}\Big((g\cdot_G h, u+\Theta(g)v+\alpha(g, h))\Big)\cdot_{\alpha}\Big(g\cdot_G h, u+\Theta(g)v+\alpha(g, h)\Big)\\
&=&\Big(\huaD(g\cdot_G h), T(u)+T(\Theta(g)v)+T(\alpha(g, h))+u+\Theta(g)v+\alpha(g, h)\\
&&-\Theta(\huaD(g\cdot_G h))(u+\Theta(g)v+\alpha(g, h))+\beta(g\cdot_G h)\Big)\cdot_{\alpha}\Big(g\cdot_G h, u+\Theta(g)v+\alpha(g, h)\Big)\\
&=&\Big(\huaD(g\cdot_G h)\cdot_G (g\cdot_G h), T(u)+T(\Theta(g)v)+T(\alpha(g, h))+u+\Theta(g)v+\alpha(g, h)\\
&&-\cancel{\Theta(\huaD(g\cdot_G h))(u+\Theta(g)v+\alpha(g, h))}+\beta(g\cdot_G h)\\
&&+\cancel{\Theta(\huaD(g\cdot_G h))(u+\Theta(g)v+\alpha(g, h))}+\alpha(\huaD(g\cdot_G h), g\cdot_G h)\Big)\\
&=&\Big(\huaD(g\cdot_G h)\cdot_G (g\cdot_G h), T(u)+T(\Theta(g)v)+T(\alpha(g, h))+u+\Theta(g)v+\alpha(g, h)\\
&&+\beta(g\cdot_G h)+\alpha(\huaD(g\cdot_G h), g\cdot_G h)\Big)
\end{eqnarray*}
and
\begin{eqnarray*}
&&\huaD_{\beta}(g, u)\cdot_{\alpha}(g, u)\cdot_{\alpha}\huaD_{\beta}(h, v)\cdot_{\alpha}(h, v)\\
&=&\Big(\huaD(g), T(u)+u-\Theta(\huaD(g))u+\beta(g)\Big)\cdot_{\alpha}(g, u)\cdot_{\alpha}\Big(\huaD(h), T(v)+v-\Theta(\huaD(h))v+\beta(h)\Big)\cdot_{\alpha}(h, v)\\
&=&\Big(\huaD(g)\cdot_G g, T(u)+u-\cancel{\Theta(\huaD(g))u}+\beta(g)+\cancel{\Theta(\huaD(g))u}+\alpha(\huaD(g), g)\Big)\\
&&\cdot_{\alpha}\Big(\huaD(h)\cdot_G h, T(v)+v-\cancel{\Theta(\huaD(h))v}+\beta(h)+\cancel{\Theta(\huaD(h))v}+\alpha(\huaD(h), h)\Big)\\
&=&\Big(\huaD(g)\cdot_G g\cdot_G\huaD(h)\cdot_G h, T(u)+u+\beta(g)+\alpha(\huaD(g), g)\\
&&+\Theta(\huaD(g)\cdot_G g)(T(v)+v+\beta(h)+\alpha(\huaD(h), h))+\alpha(\huaD(g)\cdot_G g, \huaD(h)\cdot_G h)\Big).
\end{eqnarray*}
Since $\huaD_{\beta}$ is a difference operator on the Lie group $(G\times V, \cdot_{\alpha})$, we have
\begin{eqnarray*}
0&=&-T(\alpha(g, h))-\alpha(g, h)-\beta(g\cdot_G h)-\alpha(\huaD(g\cdot_G h), g\cdot_G h)\\
&&+\beta(g)+\alpha(\huaD(g), g)+\Theta(\huaD(g)\cdot_G g)(\beta(h)+\alpha(\huaD(h), h))+\alpha(\huaD(g)\cdot_G g, \huaD(h)\cdot_G h)\\
&=&\dTD\beta(g, h)+\frkK(\alpha)(g, h),
\end{eqnarray*}
which implies that $\dTD\beta+\frkK(\alpha)=0$. Thus $\delta(\alpha, \beta)=0$, i.e. $(\alpha, \beta)$ is a $2$-cocycle.

Let $s'$ be another section and $(\alpha', \beta')$ the associated $2$-cocycle. Assume that $s'(g)=\eta(g)\cdot_\Pi s(g)$ for $\eta\in C^{1}(G, V)$. Then we have
\begin{eqnarray*}
\alpha'(g, h)-\alpha(g, h)&=&s'(g)\cdot_\Pi s'(h)\cdot_\Pi (s'(g\cdot_G h))^{-1}-s(g)\cdot_\Pi s(h)\cdot_\Pi (s(g\cdot_G h))^{-1}\\
&=&\eta(g)\cdot_\Pi s(g)\cdot_\Pi \eta(h)\cdot_\Pi s(h)\cdot_\Pi(s(g\cdot_G h))^{-1}\cdot_\Pi (\eta(g\cdot_G h))^{-1}\\
&&-s(g)\cdot_\Pi s(h)\cdot_\Pi (s(g\cdot_G h))^{-1}\\
&=&\eta(g)+s(g)\cdot_\Pi\eta(h)\cdot_\Pi(s(g))^{-1}+\cancel{(s(g)\cdot_\Pi s(h))\cdot_\Pi(s(g\cdot_\Pi h))^{-1}}\\
&&-\eta(g\cdot_G h)-\cancel{s(g)\cdot_\Pi s(h)\cdot_\Pi (s(g\cdot_G h))^{-1}}\\
&=&\eta(g)+\Theta(g)\eta(h)-\eta(g\cdot_G h)\\
&=&\dM^{\Theta}\eta(g, h),\\
\beta'(g)-\beta(g)&=&\huaD_\Pi(s'(g))\cdot_\Pi (s'(\huaD(g)))^{-1}-\huaD_\Pi(s(g))\cdot_\Pi (s(\huaD(g)))^{-1}\\
&=&\huaD_\Pi(\eta(g))\cdot_\Pi\eta(g)\cdot_\Pi\huaD_\Pi(s(g))\cdot_\Pi(\eta(g))^{-1}\cdot_\Pi(s(\huaD(g))^{-1}(\eta(\huaD(g)))^{-1}\\
&&-\huaD_\Pi(s(g))\cdot_\Pi (s(\huaD(g)))^{-1}\\
&=&T(\eta(g))+\eta(g)-\Theta(\huaD(g))\eta(g)-\eta(\huaD(g))\\
&=&\frkK(\eta)(g),
\end{eqnarray*}
which implies that $(\alpha', \beta')-(\alpha, \beta)=\delta(\eta)$. Thus, $(\alpha', \beta')$ and $(\alpha, \beta)$ are in the same cohomology class.
\end{proof}

\begin{defi}
Let $(\hat{\Pi}, \hat{\huaD}_{\hat{\Pi}})$ and $(\Pi, \huaD_\Pi)$ be two abelian extensions of a difference Lie group $(G, \huaD)$ by $(V, T)$. They are said to be {\bf isomorphic} if there exists an isomorphism $\sigma: \Pi\lon\hat{\Pi}$ of difference Lie groups such that $\sigma\circ i=\hat{i}$ and $\hat{p}\circ\sigma=p$.
%the following diagram commutes:
%\begin{equation*}
%\xymatrix@!0{\{1\}\ar@{->} [rr]&& V \ar@{->} [rr] \ar'[d] [dd] \ar@{=} [rd] && \Pi\ar'[d] [dd]\ar@ {->} [rr] \ar@{->} [rd]^{\sigma}&& G\ar@{=} [rd]\ar'[d] [dd]\ar@{->} [rr]&&\{1\}&\\
%&\{1\}\ar@{->} [rr]&& V\ar@{->} [rr]\ar@{->} [dd]&&\hat{\Pi}\ar@{->} [dd]\ar@{->} [rr]&&G\ar@ {->} [dd]\ar@{->} [rr]&&\{1\}\\
%\{1\}\ar@{->} [rr]&&V\ar'[r] [rr] \ar@{=} [rd]&&\Pi\ar@{->} [rd]^{\sigma}\ar'[r] [rr]&&G\ar@{=} [rd] \ar'[r] [rr]&&\{1\}&\\
%&\{1\}\ar@{->} [rr]&& V\ar@{->} [rr]&&\hat{\Pi}\ar@{->} [rr]&&G\ar@{->} [rr]&&\{1\}.}
%\end{equation*}
\end{defi}

\begin{thm}\label{th:abel-ext}
For a given representation $(V, T, \Theta)$ of a difference Lie group $(G, \huaD)$, abelian extensions of $(G, \huaD)$ by $(V, T)$ relative to the representation $(V, T, \Theta)$ are classified by the second cohomology group $\huaH^{2}(G, \huaD, V, T)$.
\end{thm}
\begin{proof}
Let $(\hat{\Pi}, \hat{\huaD}_{\hat{\Pi}})$ and $(\Pi, \huaD_\Pi)$ be two isomorphic abelian extensions via an isomorphism $\sigma$. Assume that $s$ is a section of $(\Pi, \huaD)$. Define $s'$ by
$$ s'=\sigma\circ s.$$
Then, it is obvious that $s'$ is a section of $(\hat{\Pi}, \hat{\huaD}_{\hat{\Pi}})$.
Note that any two isomorphic abelian extensions $(\hat{\Pi}, \hat{\huaD}_{\hat{\Pi}})$ and $(\Pi, \huaD_\Pi)$ of $(G,\huaD)$ by $(V,T)$ are relative to the same representation $(V,T,\Theta)$ of $(G,\huaD)$. Actually, for any $g\in G,\,u\in V$, we have
\begin{eqnarray*}
s(g)\cdot_{\Pi} u \cdot_{\Pi} (s(g))^{-1} &=& \sigma(s(g)\cdot_{\Pi} u \cdot_{\Pi} (s(g))^{-1})\\
&=& \sigma(s(g))\cdot_{\hat{\Pi}} \sigma(u) \cdot_{\hat{\Pi}} (\sigma(s(g)))^{-1}\\
&=& s'(g)\cdot_{\hat{\Pi}} u \cdot_{\hat{\Pi}} (s'(g))^{-1},
\end{eqnarray*}
which implies that the  representation $(V,T,\Theta)$ given in Proposition \ref{pro-repn} are the same.

Denote by $(\alpha, \beta)$  and $(\alpha', \beta')$ the corresponding $2$-cocycle given in Theorem \ref{thmabeH2} respectively. Then we have
\begin{eqnarray*}
\alpha'(g, h)&=&s'(g)\cdot_{\hat{\Pi}}s'(h)\cdot_{\hat{\Pi}}(s'(g\cdot_{\hat{\Pi}} h))^{-1}\\
&=&\sigma(s(g))\cdot_{\hat{\Pi}}\sigma(s(h))\cdot_{\hat{\Pi}}(\sigma(s(g\cdot_{\hat{\Pi}} h)))^{-1}\\
&=&\sigma(\alpha(g, h))\\
&=&\alpha(g, h).
\end{eqnarray*}
Similarly, we have $\beta'=\beta$. By Theorem \ref{thmabeH2}, isomorphic abelian extensions give rise to the same cohomological class in $\huaH^{2}(G, \huaD, V, T)$.

Conversely, given a $2$-cocycle $(\alpha, \beta)$,  we define a group multiplication $\cdot_\alpha$ on $G\times V$ by
\begin{equation*}
(g, u)\cdot_{\alpha}(h, v)=(g\cdot_G h, u+\Theta(g)v+\alpha(g, h)),\quad \forall g, h\in G, u, v\in V.
\end{equation*}
By $\dM^{\Theta}\alpha=0$, it is straightforward to deduce that $(G\times V, \cdot_{\alpha})$ is a Lie group. Define a smooth map $\huaD_{\beta}: G\times V\lon G\times V$ by
\begin{equation}\label{eq:db}
\huaD_{\beta}(g, u)=(\huaD(g), T(u)+u-\Theta(\huaD(g))u+\beta(g)), \quad \forall g\in G, u\in V.
\end{equation}
Since $\dTD\beta+\frkK(\alpha)=0$, it is straightforward to deduce that $\huaD_{\beta}$ is a difference operator on the difference Lie group $(G\times V, \cdot_{\alpha})$. Thus $(G\times V, \cdot_{\alpha}, \huaD_{\beta})$ is a difference Lie group, which is an abelian extension of $(G, \huaD)$ by $(V, T)$.

Choose another $2$-cocycle $(\alpha', \beta')$, such that $(\alpha, \beta)$ and $(\alpha', \beta')$ are in the same cohomology class, i.e.
$$(\alpha-\alpha', \beta-\beta')=\delta(\eta)=(\dM^{\Theta}\eta,\frkK(\eta)), $$
where $\eta\in C^{1}(G, V)$, and denote the corresponding difference Lie group by $(G\times V, \cdot_{\alpha'}, \huaD_{\beta'})$. Define a smooth map $\sigma:G\times V\lon G\times V$ by
$$
\sigma(g, u)=(g, u+\eta(g)),
$$
for all $g\in G, u\in V$. Since
\begin{eqnarray*}
\sigma((g, u)\cdot_{\alpha}(h, v))&=&\sigma(g\cdot_G h, u+\Theta(g)v+\alpha(g, h))\\
&=&(g\cdot_G h, u+\Theta(g)v+\alpha(g, h)+\eta(g\cdot_G h))
\end{eqnarray*}
and
\begin{eqnarray*}
\sigma((g, u))\cdot_{\alpha'}\sigma((h, v))&=&(g, u+\eta(g))\cdot_{\alpha'}(h, v+\eta(h))\\
&=&(g\cdot_G h, u+\eta(g)+\Theta(g)v+\Theta(g)\eta(h)+\alpha'(g, h)),
\end{eqnarray*}
we deduce that $\sigma$ is a Lie group isomorphism. Moreover, we have
\begin{eqnarray*}
\huaD_{\beta'}(\sigma(g, u))&=&\huaD_{\beta'}(g, u+\eta(g))\\
&=&(\huaD(g), T(u)+T(\eta(g))+u+\eta(g)-\Theta(\huaD(g))u-\Theta(\huaD(g))\eta(g)+\beta'(g)),
\end{eqnarray*}
and
\begin{eqnarray*}
\sigma(\huaD_{\beta}(g, u))&=&\sigma(\huaD(g), T(u)+u-\Theta(\huaD(g))u+\beta(g))\\
&=&(\huaD(g), T(u)+u-\Theta(\huaD(g))u+\beta(g)+\eta(\huaD(g))).
\end{eqnarray*}
Since $\beta(g)-\beta'(g)=T(\eta(g))+\eta(g)-\Theta(\huaD(g))\eta(g)-\eta(\huaD(g))$, we have $\huaD_{\beta'}\circ\sigma=\sigma\circ\huaD_{\beta}$. Therefore, $\sigma$ is a difference Lie group isomorphism. Then it is straightforward to deduce that the two abelian extensions are isomorphic.
\end{proof}

  Note that for any
abelian extension $(\Pi,\huaD_\Pi)$ of $(G, \huaD)$ by $(V, T)$ relative to the representation $(V, T, \Theta)$, the Lie group $\Pi$ is isomorphic to the semidirect product Lie group $G\ltimes_\Theta V$, if and only if there exists a section $s$ of $(\Pi,\huaD_\Pi)$ being a group homomorphism from $G$ to $\Pi$. Indeed, for the $2$-cocycle $(\alpha, \beta)$ obtained by such a section $s$, we clearly have $\alpha=0$ by Eq.~\eqref{2-cocycle}.

In this situation, we have the following result as a byproduct of Theorem~\ref{cohomology-exact-DG} and Theorem~\ref{th:abel-ext}.
\begin{cor}
For a given representation $(V, T, \Theta)$ of a difference Lie group $(G, \huaD)$,
the difference operators on the semidirect product Lie group $G\ltimes_\Theta V$ are classified by the quotient $\huaH^{2}(\huaD, T)/\frkk^1(\huaH^{1}(G,V))$, where $\frkk^1$ is given by \eqref{connmap}.
\end{cor}
\begin{proof}
%Assume that $(\Pi,\huaD_\Pi)$ is an abelian extension of $(G, \huaD)$ by $(V, T)$ relative to the representation $(V, T, \Theta)$, such that $\Pi$ is isomorphic to $G\ltimes_\Theta V$. Then by Theorem~\ref{th:abel-ext}, there exist two 2-cocycles $(\alpha, \beta)$ and $(0, \beta')$ in the same cohomology class, i.e.
%$$(\alpha, \beta-\beta')=\delta(\eta)=(\dM^{\Theta}\eta,\frkK(\eta))$$
%for $\eta\in C^{1}(G, V)$.
%
 Let  $(G\ltimes_\Theta V,\huaD_\beta)$ and $(G\ltimes_\Theta V,\huaD_{\beta'})$ be two abelian extensions associated with 2-cocycles $(0, \beta)$ and $(0, \beta')$ respectively in $C^{2}(G, \huaD, V, T)$, where $\huaD_\beta$ is given by \eqref{eq:db}. It is obvious that $\beta$ and $\beta'$ are 2-cocycles in $\frkC^{2}(\huaD, T)$.  By Theorem~\ref{th:abel-ext}, if $(G\ltimes_\Theta V,\huaD_\ltimes)$ and $(G\ltimes_\Theta V,\huaD'_\ltimes)$ are isomorphic, then $(0, \beta)$ and $(0, \beta')$ are in the same cohomology class, i.e.
$$(0, \beta-\beta')=\delta(\eta)=(\dM^{\Theta}\eta,\frkK(\eta))$$
for $\eta\in C^{1}(G, V)$. Therefore,  $\dM^{\Theta}\eta=0$   and $\beta-\beta'=\frkK(\eta)$. This means that $[\beta]=[\beta']$ in $\huaH^{2}(\huaD, T)/\frkk^1(\huaH^{1}(G,V))$.

Conversely, for any 2-cocycle $\beta$ in $\frkC^{2}(\huaD, T)$, define $\huaD_\beta$   by \eqref{eq:db}. Then $(G\ltimes_\Theta V,\huaD_\beta)$ is a difference Lie group. Let $\beta'$ be another 2-cocycle such that
$$[\beta]-[\beta']=[\frkK(\eta)]=\frkk^1([\eta]),$$
for some 1-cocycle $\eta$   in $C^{1}(G, V)$.  Denote by  $(G\ltimes_\Theta V,\huaD_{\beta'})$ the corresponding difference Lie group. Then the smooth map $\sigma:G\times V\lon G\times V$ defined by
$$
\sigma(g, u)=(g, u+\eta(g))
$$
gives rise to an isomorphism from the difference Lie group $(G\ltimes_\Theta V,\huaD_\beta)$ to the difference Lie group $(G\ltimes_\Theta V,\huaD_{\beta'})$.
%$\frkK(\eta)$ is a 2-cocycle in $\frkC^{2}(\huaD, T)$, as we have known that $\dTD\frkK+\frkK\dM^{\Theta}$ by Theorem~\ref{th:coboundary}. Hence, the difference operators on $G\ltimes_\Theta V$ are classified by $Z^{2}(\huaD, T)/\frkK(Z^{1}(G,V))$.It is equivalent to say that

As a result, the difference operators on $G\ltimes_\Theta V$ are classified by
$\huaH^{2}(\huaD, T)/\frkk^1(\huaH^{1}(G,V))$.
\end{proof}

Further, if there exists a section $s$ of $(\Pi,\huaD_\Pi)$ being a difference Lie group homomorphism from $(G,\huaD)$ to $(\Pi,\huaD_\Pi)$, then the corresponding $2$-cocycle $(\alpha, \beta)=0$ by Eqs.~\eqref{2-cocycle}, \eqref{2-cocycle'}, and $(\Pi,\huaD_\Pi)$ is isomorphic to the semidirect product difference Lie group $(G\ltimes_\Theta V,\Gamma)$ defined in Theorem~\ref{th:semidirect}.

%\section{Matched pairs of difference Lie groups}


\begin{thebibliography}{a}

\bibitem{AC} C. Abad and M. Crainic, The Weil algebra and the Van Est isomorphism. \emph{Ann. Inst. Fourier (Grenoble)} {\bf 61} (3), (2011), 927-970.

\bibitem{AGV} I. Angiono, C. Galindo and L. Vendramin, Hopf braces and Yang-Baxter operators, \textit{Proc. Amer. Math. Soc.} {\bf 145} (2017), 1981-1995.

%\bibitem{CC}
%R. Caseiro and J. Costa, $\huaO$-operators on Lie $\infty$-algebras with respect to Lie $\infty$-actions. arXiv:2109.01363. \emph{Comm. Algebra} https://doi.org/10.1080/00927872.2022.2025819.

\bibitem{CD}
A. Cabrera and T. Drummond, Van Est isomorphism for homogeneous cochains. \emph{Pacific J. Math.} {\bf 287} (2) (2017), 297-336.

\bibitem{CS}A. Calderon and N. Salter, Relative homological representations of framed mapping class groups. \emph{Bull. London Math. Soc.} {\bf53} (2021), 204-219.

\bibitem{Cr}
M. Crainic, Differentiable and algebroid cohomology, van Est isomorphisms, and characteristic classes. \emph{Comment. Math. Helv.} {\bf 78} (4)  (2003), 681-721.
%\bibitem{Da} H. G. Dales, Banach Algebras and Automatic Continuity. Oxford University Press, 2000.

%\bibitem{Das20}
%A. Das, Deformations of associative Rota-Baxter operators. \emph{J. Algebra} {\bf560} (2020), 144-180.
%
%\bibitem{Das}A. Das, Leibniz algebras with derivations. \emph{J. Homotopy Relat. Struct.} {\bf 16} (2021), no. 2, 245-274.
%
%\bibitem{DM} A. Das and A. Mandal, Extensions, deformations and categorifications of AssDer pairs. arXiv: 2002.11415.

%\bibitem{DR} V. A. Dolgushev and C. L. Rogers, A version of the Goldman-Millson Theorem for filtered $L_{\infty}$-algebras. \emph{J. Algebra} {\bf 430} (2015), 260-302.

%\bibitem{DL} M. Doubek and T. Lada, Homotopy derivations. \emph{J. Homotopy Rela. Struct.} {\bf 11}  (2016), no. 3, 599-630.

\bibitem{EM}
S. Eilenberg and S. MacLane, Cohomology theory in abstract groups. I. \emph{Ann. of Math. (2)} {\bf 48} (1947), 51-78.

%\bibitem{Fre}
%Y. Fr\'egier, A new cohomology theory associated to deformations of Lie algebra morphisms. \emph{Lett. Math. Phys.} {\bf 70} (2004), no. 2, 97-107.

%\bibitem{Fregier-Zambon-1}
%Y. Fr\'egier and M. Zambon, Simultaneous deformations and Poisson geometry. \emph{Compos. Math.} {\bf 151} (2015), 1763-1790.

%\bibitem{Get}E. Getzler, Lie theory for nilpotent $L_{\infty}$-algebras. \emph{Ann. Math. (2)} {\bf 170} (2009), 271-301.

\bibitem{GK}
L. Guo and W. Keigher, On differential Rota-Baxter algebras. \emph{J. Pure Appl. Algebra} {\bf 212} (2008), 522-540.

\bibitem{GLS} L. Guo, H. Lang and Y. Sheng, Integration and geometrization of Rota-Baxter Lie algebras. \textit{Adv. Math.} {\bf 387} (2021), 107834.

%\bibitem{GLST}
%L. Guo, Y. Li, Y. Sheng and R. Tang, Crossed homomorphisms and Cartier-Kostant-Milnor-Moore theorem for difference Hopf algebras. arXiv: 2112.08434v1.

%\bibitem{GLSZ}
%L. Guo, Y. Li, Y. Sheng and G. Zhou, Cohomologies, extensions and deformations of differential algebras with any weights. arXiv: 2003.03899.

%\bibitem{GSZ}
%L. Guo, W. Sit and R. Zhang, Differential type operators and Grobner-Shirshov bases. \emph{J. Symbolic Computation} {\bf 52} (2013), 97-123.

%\bibitem{GD}S. Guo and A. Das, On $3$-Lie algebras with a derivation. arXiv: 2110. 04715.

\bibitem{HN}
J. Hilgert and  K.-H. Neeb, Structure and geometry of Lie groups. Springer Monographs in Mathematics. Springer, New York, 2012.


\bibitem{Hou}
J. C. Houard, An integral formula for cocycles of Lie groups. \emph{Ann. Inst. H. Poincare Sec. A (N.S.)} {\bf 32} (3) (1980), 221-247.

\bibitem{JS}
J. Jiang and Y. Sheng, Deformations, cohomologies and integrations of relative difference Lie algebras.   \emph{J. Algerba} {\bf614} (2023),  535-563.


\bibitem{Kasahara}
Y. Kasahara, Crossed homomorphisms and low dimensional representations of mapping class groups of surfaces. arXiv:2202.06368.



%\bibitem{LST}A. Lazarev, Y. Sheng and R. Tang, Deformations and homotopy theory  of relative Rota-Baxter Lie algebras.  \emph{Comm. Math. Phys.} {\bf383} (2021),  595-631.

%\bibitem{Kol} E. R. Kolchin, {Differential Algebras and Algebraic Groups}. Academic Press, New York, 1973.

\bibitem{Lev} A. Levin, {Difference Algebra}. Springer, 2008.

%\bibitem{Lod} J. L. Loday, On the operad of associative algebras with derivation. \emph{Georgian Math. J.} {\bf 17}  (2010), no. 2, 347-372.

\bibitem{Li}
D. Li-Bland and E. Meinrenken, On the van Est homomorphism for Lie groupoids. \emph{Enseign. Math.} {\bf 61} (1-2) (2015), 93-137.

\bibitem{LGG} X. Liu, L. Guo and X. Guo, $\lambda$-Differential operators and $\lambda$-differential modules for the Virasoro algebra. {\em Linear  Multilinear Algebra} {\bf 67} (2019), no. $7$, 1308-1324.

\bibitem{Lue}
A. Lue, Crossed homomorphisms of Lie algebras. \emph{Proc. Cambridge Philos. Soc.} {\bf 62} (1966), 577-581.

\bibitem{MS}
E. Meinrenken and M. A. Salazar, Van Est differentiation and integration. \emph{Math. Ann.} {\bf 376} (3-4)  (2020), 1395-1428.

\bibitem{MQ}
I. Mencattini and A. Quesney, Crossed homomorphisms, integration of Post-Lie algebras and the Post-Lie magnus expansion. \emph{Comm. Algebra} {\bf 49} (2021), no. $8$, 3507-3533.


\bibitem{Ne} K.-H. Neeb, Abelian extensions of infinite-dimensional Lie groups. \emph{Trav. Math.} {\bf 15} (2004), 69-194.

%\bibitem{NR} A. Nijenhuis  and R. Richardson,  Cohomology and deformations in graded Lie algebras. {\em Bull. Amer. Math. Soc.} {\bf 72} (1966), 1-29.

%\bibitem{NR2} A. Nijenhuis and R. Richardson,  Commutative  algebra cohomology and deformations of Lie and associative algebras. {\em J. Algebra} {\bf 9} (1968), 42-105.

%\bibitem{Pfl}
%M. J. Pflaum, H. Posthuma and X. Tang, The localized longitudinal index theorem for Lie groupoids and the van Est map. \emph{Adv. Math.} {\bf 270} (2015), 223-262.

\bibitem{PSTZ}
Y. Pei, Y. Sheng, R. Tang and K. Zhao, Actions of monoidal categories and representations of Cartan type Lie algebras. \emph{J. Inst. Math. Jussieu} (2022), https://doi.org/10.1017/S147474802200007X.

\bibitem{PPT} M.~J. Pflaum, H. Posthuma and X. Tang, The localized longitudinal index theorem for Lie groupoids and the van Est map, \emph{Adv. Math.} {\bf 270} (2015), 223-262.

\bibitem{RS} B. Rangipour and S. Sutlu, A van Est isomorphism for bicrossed product Hopf algebras, \emph{Comm. Math. Phys.} {\bf 311} (2012), 491-511.

\bibitem{RS1} A.~G. Reyman and M.~A. Semenov-Tian-Shansky, Reduction of Hamiltonian systems, affine Lie algebras and Lax equations. \emph{Invent. Math.} {\bf 54} (1979), 81-100.

\bibitem{STS}
M. A. Semenov-Tian-Shansky, What is a classical $r$-matrix? \emph{Funct. Anal. Appl.} {\textbf 17}, (1983), 259-272.

\bibitem{Se} J. P. Serre, Galois Cohomology. Springer, 1997.

%\bibitem{SW} Q. Sun and Z. Wu, Cohomologies of $n$-Lie algebras with derivations. \emph{Mathematics} {\bf 2021} 9, 2452.

%\bibitem{Sw} M. E. Sweedler, Cohomology of algebras over Hopf algebras. \textit{Trans. Amer. Math. Soc.} \textbf{133} (1968), 205-239.

%\bibitem{TBGS}R. Tang, C. Bai, L. Guo and Y. Sheng, Deformations and their controlling cohomologies of $\mathcal{O}$-operators. \emph{Comm. Math. Phys.} {\bf368} (2019), 665-700.

%\bibitem{TFS}
%R. Tang, Y. Fr\'egier and Y. Sheng, Cohomologies of a Lie algebra with a derivation and applications. \emph{J. Algebra} {\bf 534} (2019) 65-99.

\bibitem{Tsang} C. Tsang, Non-existence of Hopf-Galois structures and bijective crossed homomorphisms.  \emph{J. Pure Appl. Algebra} {\bf 223} (2019), no. 7, 2804-2821.

%\bibitem{PS} M. van der Put and M. Singer, {Galois Theory of Linear Differential Equations}. Springer-Verlag, 2003.

%\bibitem{PS1} M. van der Put and M. Singer, Galois Theory of Difference Equations. Springer, 1997.

\bibitem{Van}  W. T. van Est, Group cohomology and Lie algebra cohomology in Lie groups. I, II.  \emph{Nederl. Akad. Wetensch. Proc. Ser. A. 56 = Indagationes Math.} \textbf{15} (1953).

%\bibitem{Vo}T. Voronov, Higher derived brackets and homotopy algebras. \emph{J. Pure Appl. Algebra}  {\bf 202} (2005), 133-153.

%\bibitem{WZ}   K. Wang and G. Zhou, Deformations and homotopy theory of Rota-Baxter algebras of any weight. arXiv:2108.06744.

%\bibitem{WX}
%A. Weinstein and P. Xu, Extensions of symplectic groupoids and quantization. \emph{J. Reine. Angew. Math.} \text{\bf 417} (1991), 159-189.
%\bibitem{Wei} C. A. Weibel, An Introduction to Homological Algebra.  Cambridge Studies in Advanced Mathematics, 38. \emph{ Cambridge University Press, Cambridge.} xiv+450 pp, (1994).

\bibitem{Whi} J. H. C. Whitehead, Combinatorial Homotopy II. {\em Bull. Amer. Math. Soc.} {\bf 55} (1949), 453-496.

%\bibitem{XL} S. Xu and J. Liu, Cohomologies of $3$-Lie algebras with derivations. arXiv: 2110. 04215.

\end{thebibliography}
 \end{document}